\documentclass[12pt]{amsart}
\usepackage{bibunits}
\usepackage{amscd}
\usepackage{verbatim}
\usepackage{pdfpages}
\usepackage[mathscr]{euscript}
\usepackage{enumerate}

\usepackage{amssymb, amsmath}

\DeclareMathAlphabet{\cat}{OT1}{cmss}{m}{sl}

\newtheorem{theorem}{Theorem}[section]

\newtheorem{proposition}[theorem]{Proposition}
\newtheorem{lemma}[theorem]{Lemma}
\newtheorem{corollary}[theorem]{Corollary}

\theoremstyle{definition}
\newtheorem{remark}[theorem]{Remark}

\newtheorem{example}[theorem]{Example}

\newtheorem{notation}[theorem]{Notation}

\newcommand{\sep}{\mathrm{sep}}

\title[On the Tits--Weiss conjecture and the Kneser--Tits conjecture]{On the Tits--Weiss conjecture
and the Kneser--Tits conjecture for $\mathrm{E}^{78}_{7,1}$ and $\mathrm{E}^{78}_{8,2}$ } 

\author
[S. Alsaody]
{Seidon Alsaody}
\thanks{S.\ Alsaody wishes to thank PIMS for their financial support.}

\author
[V. Chernousov]
{Vladimir Chernousov}
\thanks{ V.\ Chernousov was partially supported by  Alexander von Humboldt Foundation 
and an NSERC research grant}

\author
[A. Pianzola]
{Arturo Pianzola \\ \\ with an appendix by Richard M.\ Weiss}
\thanks{A.\ Pianzola wishes to thank NSERC and CONICET for their
continuous support}

\address
{Department of Mathematical Sciences, University of Alberta,
Edmonton, Alberta, Canada T6G 2G1}

\email{alsaody@ualberta.ca}

\address
{Department of Mathematical Sciences, University of Alberta,
Edmonton, Alberta, Canada T6G 2G1}

\email {vladimir@ualberta.ca}

\address
{Department of Mathematical Sciences, University of Alberta,
Edmonton, Alberta, Canada T6G 2G1}

\email{a.pianzola@ualberta.ca}

\begin{document}
\begin{bibunit}
\begin{abstract} We prove that the structure group of any Albert algebra over an arbitrary field  is $R$-trivial.  This implies the  Tits--Weiss conjecture
 for Albert algebras and the Kneser--Tits conjecture for
isotropic groups of type $\mathrm{E}_{7,1}^{78}, \mathrm{E}_{8,2}^{78}$. As a further corollary, we show that
some standard conjectures on the groups of $R$-equivalence classes  in algebraic groups 
and the norm principle are true for strongly inner forms 
of type $^1\mathrm{E}_6$.
\end{abstract}

\maketitle

\section{Introduction}

The primary aim of this paper is to prove the long standing 
Tits--Weiss conjecture on $U$-operators in Albert algebras and the Kneser--Tits
conjecture for algebraic groups of type $\mathrm{E}^{78}_{7,1}$ and $\mathrm{E}^{78}_{8,2}$.

The Tits--Weiss conjecture asserts that the structure group ${\rm Str}(A)$ 
of an arbitrary Albert algebra $A$ 
is generated by the inner structure group, formed by the so-called $U$-operators, and the central homotheties. This problem 
was raised  by Tits and Weiss 
in their 2002 book \cite{TiW}, where they studied spherical buildings and the corresponding generalized polygons attached to isotropic groups of relative rank $2$.  
Despite many efforts, this problem has remained out of reach.

If $G$ is an isotropic  simple simply connected group over $K$ of relative rank $\geq 2$
then by \cite{PrasRag} the group $G(K)$ is generated by $K$-points of isotropic subgroups
of $G$ of relative rank $1$. This result allows to reduce many problems for $G(K)$ to groups of relative rank $1$. For instance, this is the case for the Kneser-Tits
problem (see below). 
Note also that isotropic groups of relative rank $1$ give rise to important examples
of more general  groups of rank one.
The latter were introduced by Tits in the early 1990s, 
who called them Moufang sets. They have proved to be important in the classification of simple groups, incidence geometry, the
theory of buildings, and other areas.
Further still, rank one groups are useful in studying 
isotropic groups of exceptional types, where algebraic groups and their associated root subgroups 
are typically parametrized 
by a nonassociative structure, and, as emphasized in \cite{MMS},
a rich interplay emerges between rank one groups, nonassociative algebras, and linear algebraic groups.

The Kneser--Tits conjecture for a simple simply connected isotropic group ${\bf G}$ over a field $K$ asserts
that the abstract group ${\bf G}(K)$ of $K$-points of ${\bf G}$ coincides with its normal subgroup 
${\bf G}(K)^+$  generated by the unipotent 
radicals of the minimal parabolic $K$-subgroups of ${\bf G}$. We refer to \cite{gille} for
a survey of the history and recent results on this conjecture.
Its importance 
comes from the fact that the group ${\bf G}(K)^+$ has a  natural $BN$-pair structure and hence 
 is projectively simple (i.e.\ simple modulo its centre), by a celebrated theorem of Tits. 
So if ${\bf G}(K)={\bf G}(K)^+$
we would have many more new examples of  projectively simple abstract groups
given by $K$-points of isotropic simple simply connected algebraic groups. In this way, we would obtain analogues of finite simple groups of Lie type in the case 
of infinite fields.
It is also worth mentioning that the information about the normal subgroup structure of ${\bf G}(K)$ is crucial
in the arithmetic of algebraic groups for studying, among other things, congruence subgroups, discrete subgroups,
 lattices, and locally symmetric spaces.
In general, the Kneser--Tits conjecture does not hold, and the first counterexample was constructed 
by V.\ Platonov in 1975 \cite{Pl}. However, it is believed by specialists that 
the conjecture holds for many isotropic groups of exceptional type, including those of type $\mathrm{E}_{7,1}^{78}$
and $\mathrm{E}_{8,2}^{78}$. 

The bridge connecting the Tits--Weiss conjecture and the Kneser--Tits conjecture
for the abovementioned forms of type $\mathrm{E}_7$ and $\mathrm{E}_8$ is provided by a theorem of
Tits and Weiss (see the Appendix), which states that the two conjectures are equivalent.\footnote{
The proof of the equivalence of the two conjectures is not straightforward, and no direct reference is available.
 We  are grateful  to Richard Weiss for writing a detailed proof of this equivalence. It is included as an appendix to this paper.}
 It is interesting to mention that the proof in the Appendix is characteristic free.
Furthermore, using  P.\ Gille's results on Whitehead groups in \cite{gille} on Whitehead groups, one can easily see that 
\begin{enumerate}
 \item the Kneser--Tits conjecture for the abovementioned groups reduces to the
$R$-triviality of structure groups of Albert algebras, and
\item the conjecture holds in arbitrary characteristic once it is established in characteristic zero.
\end{enumerate}

Our main result is the following.

\smallskip

\noindent
{\bf Theorem.} {\it Let $A$ be an Albert algebra over a field $K$. Then the structure group ${\bf Str}(A)$ of $A$ is $R$-trivial, i.e.
for any field extension $F/K$ the group of $R$-equivalence classes ${\bf Str}(A)(F)/R$
is trivial.}

\smallskip

As explained above, this implies that the Tits--Weiss conjecture on $U$-opera\-tors holds
for Albert algebras over any field, and that the same
is true for the Kneser--Tits conjecture for groups of type $\mathrm{E}^{78}_{7,1}$ and $\mathrm{E}^{78}_{8,2}$.
Our proof is of a geometric nature. We carefully analyze   the properties
of the natural action of the structure group ${\bf Str}(A)$ on the corresponding Albert algebra
$A$.
The information
that we need is encoded in the Galois cohomology of the  stabilizers of subalgebras of $A$.
We compute the Galois cohomology of all these stabilizers and using this information, we explicitly construct a system of generators 
of ${\bf Str}(A)(K)$, which we prove is $R$-trivial.  

\smallskip

In the course of the proof of the above theorem we reprove a result due to S.\ Garibaldi and H.\ Petersson
on the weak Skolem--Noether Theorem problem  for isomorphic embeddings (for the terminology,
definitions and the precise statement we refer to \cite{GP}). Note that the more general  
 Skolem--Noether problem for isotopic (as opposed to isomorphic) embeddings is still out of reach.
In their joint paper \cite{GP}, S.\ Garibaldi and H.\ Petersson write:

\smallskip

\noindent
{\it Regrettably, the methodological arsenal assembled in the present paper, consisting 
as it does of rather elementary manipulations  involving the two Tits constructions, does not
seem strong enough to provide an affirmative answer to this question.}

\smallskip

Our proof of the Skolem--Noether theorem for isomorphic embeddings is much shorter and is based 
on torsor techniques. It seems plausible that it can be also used 
for the proof of more difficult Skolem--Noether problem for isotopic embeddings.
For that reason, we include our alternative proof in Section \ref{SN} below.

\smallskip

We would also like to mention that it follows directly from our main theorem that two standard conjectures hold for simple simply connected strongly inner forms
of type $\mathrm{E}_6$: the abelian nature of the group of $R$-equivalence classes,
and the existence of transfers for the functor of $R$-equivalence classes. For these groups, the norm principle holds as well.
For the details we refer to the last section of the paper.

\smallskip

Lastly we mention that when the work on this paper was completed, M.\ Thakur posted a 
preprint \cite{Thakur4} on arXiv where he proves the same result. His proof differs from
our and is based on some explicit formulas for
automorphisms of subalgebras of Albert algebras and extensions of these automorphisms.

\smallskip

\noindent
{\it Acknowledgements.} We are grateful to H.\ Petersson, A.\ Stavrova and R.\ Weiss for fruitful discussions 
and comments.  

\subsection*{Convention} In view of point (2) above, it suffices, in order to prove our main results over arbitrary fields, 
to work over a base field of characteristic zero. In view of the possible independent interest of partial results, we will be less restrictive. We therefore fix a field
$K$ which, unless otherwise stated, is infinite and of characteristic different from 2 and 3.

\section{Preliminaries}\label{preliminaries}
For later use we record some facts about Albert algebras and algebraic groups. 

\subsection{Albert Algebras}
A \emph{Jordan algebra} over $K$ is a unital, commutative, not necessarily associative $K$-algebra $A$ in which the Jordan identity
\[(xy)(xx)=x(y(xx))\]
is satisfied. In particular, $A$ is power associative. Given an associative algebra $B$ with multiplication $\cdot$, the anticommutator $\frac{1}{2}(x\cdot y+y\cdot x)$ 
 defines on $B$ the structure of a Jordan algebra, denoted by 
 $B^+$.
 A Jordan algebra $A$ is said to be \emph{special} if it is isomorphic to a Jordan 
subalgebra of $B^+$
for an associative algebra $B$, and \emph{exceptional} otherwise.
An \emph{Albert algebra} is by definition a simple, exceptional Jordan algebra. 
It is known that any Albert algebra has dimension 27, and, if the field $K$ is separably closed, then
all Albert algebras over $K$ are isomorphic, as follows from \cite[37.11]{KMRT}.
Thus over an arbitrary field $K$ all Albert algebras are twisted forms 
of each other.

The automorphism group ${\bf H}={\bf Aut}(A)$ of an Albert algebra $A$ is a simple algebraic group  over $K$ 
of type ${\rm F}_4$. It is known that any such group arises in this fashion, and that two Albert algebras are isomorphic if and only if their automorphism groups are.
Moreover, $A$ is equipped with a cubic form $N:A\to K$, called the \emph{norm} of $A$. 
Our main object of study is the \emph{structure group} ${\bf Str}(A)$ of $A$, which is an algebraic group. For any $K$-ring $R$ we have
\[{\bf Str}(A)(R)=\{ x\in {\bf GL}(A)(R)\ | \ N_R(x(a))=\nu(x) N_R(a)\ \ \forall \,a\in A\otimes R\ \},\]
where $N_R$ is the base change of $N$ to $A\otimes R$, and the \emph{multiplier} $\nu(x)\in R^\times$ depends only on $x$.

The derived subgroup $${\bf G}=[{\bf Str}(A),{\bf Str}(A)]$$ of ${\bf Str}(A)$ 
is known to be a strongly inner form of a split  simple simply connected algebraic group
of type ${\rm E}_6$. Moreover ${\bf Str}(A)$ is an almost direct product of ${\bf G}_m$ and ${\bf G}$ (the intersection being the centre
of $\bf{G}$). Hence $$\overline{{\bf G}}={\bf Str}(A)/{\bf G}_m$$ is an adjoint group of type ${\rm E}_6$.
Since in the split, and even isotropic, case, ${\bf G}$  and ${\bf Str}(A)$ are 
$R$-trivial (see \cite{ChPl}), we may, in the proof of the main theorem, assume
that $\bf G$ is $K$-anisotropic. This amounts to the Albert algebra $A$ being a division 
algebra, which is equivalent to 
the norm map $N$ being anisotropic, i.e. 
the equation $N(a)=0$ having no non-zero solutions over $K$. 

We denote the group of $K$-points of ${\bf Str}(A)$ (resp. ${\bf G},\, {\overline{\bf G}},\, {\bf H}$) by
${\rm Str}(A)$ (resp. $G,\,\overline{G},\,H$). The group ${\bf H}$ coincides with the stabilizer in ${\bf Str}(A)$ of  $1\in A$; see e.g.\ \cite[5.9.4]{SV}.

\subsection{Subgroups of Type $\mathrm D_4$}
First we recall a statement about groups of type $\mathrm D_4$ 
inside split groups of type $\mathrm{F}_4$, proved in \cite{ACP}.
 
\begin{proposition}\label{splitD_4} Let ${\bf H}$ be a split group of type $\mathrm{F}_4$ over $K$.
Then any $K$-subgroup ${\bf M}\subset {\bf H}$  of type $\mathrm{D}_4$ is quasi-split.
\end{proposition}

For later use we need one more fact about groups of type $^{3,6}\mathrm{D}_4$ inside $\mathrm{F}_4$.  Let thus
${\bf M}\subset {\bf H}$ be a split subgroup of type $\mathrm{D}_4$, and consider its normalizer ${\bf N}=N_{\bf H}({\bf M})$. The quotient
group ${\bf N}/{\bf M}$ is isomorphic to the group of outer automorphisms ${\bf Out}({\bf M})$ of ${\bf M}$.
This is the symmetric group $S_3$, which we view as a constant finite group scheme over $K$.

Let $ [\xi]\in H^1(K,N)$ be an arbitrary cohomology class, and consider its image $[\overline{\xi}]\in
H^1 (K,{\bf Out}({\bf M}))$. Since ${\bf Out}({\bf M})$ is a constant group scheme, any cocycle
$\overline{\xi}$ representing it corresponds to a homomorphism $\phi_{\xi}:{\rm Gal}(K^{sep}/K) \to S_3$.
The image ${\rm Im}\,\phi_{\xi}$ is then isomorphic to the Galois group of the minimal Galois extension
$F/K$ over which 
the twisted group $^{\xi}{\bf M}$ becomes a group of inner type. It follows that ${\rm Im}\,\phi_{\xi}$
is generated by the cycle $(123)$ if  $^{\xi}{\bf M}$ has type $^3\mathrm{D}_4$, and is equal to $S_3$ 
 if $^{\xi}{\bf M}$ has type $^6\mathrm{D}_4$. 

\begin{lemma}\label{twistingD_4}
 Assume that $^{\xi}{\bf M}$ has type $^{3,6}\mathrm{D}_4$. Then the natural map 
${\bf ^{\xi}N}(K) \to {^{\overline{\xi}}(S_3)(K)}$ is surjective.
\end{lemma}

\begin{proof} 
If ${\rm Im}\,\phi_{\xi}=S_3$ we have $^{\overline{\xi}}(S_3)(K)=1$ and there is noting to prove.
Assume next that ${\rm Im}\,\phi_{\xi}=\langle(123)\rangle$. Then the group $^{\overline{\xi}}(S_3)(K)$ consists of 3 elements.
On the other hand, by \cite{GP}, the group $^{\xi}{\bf M}$, being of type $^3\mathrm{D}_4$, has outer automorphisms over $K$. \footnote{Strictly speaking, it is proved 
in \cite{GP} that any adjoint group of type $^3\mathrm{D}_4$ has  
an outer $K$-automorphism. However, it is easy to show that any such $K$-automorphism can be lifted, modulo inner automorphisms, to the corresponding simply connected group.}
 The image of any such automorphism (viewed as an element in ${\bf ^{\xi}N}(K)$)  under
the natural map generates $^{\overline{\xi}}(S_3)(K)$, and the assertion follows.
\end{proof}

\subsection{$R$-equivalence in Algebraic Groups}

Let ${\bf G}$ be an affine algebraic group defined over $K$.
Two $K$-points $x,y\in {\bf G}(K)$ are called \emph{ $R$-equivalent} if there is a path from 
$x$ to $y$, i.e. if there exists a rational 
map $f:\mathbb A^1\dashrightarrow {\bf G}$ defined at $0,1$ and mapping
$0$ to $x$ and $1$ to $y$. One can easily verify that this is indeed an equivalence 
relation on ${\bf G}(K)$, and that moreover, ${\bf G}$ induces a group structure on the set ${\bf G}(K)/R$ of
all $R$-equivalence classes. We will denote the set of elements 
in ${\bf G}(K)$ equivalent to $1$ by $R{\bf G}(K)$ throughout.

\subsection{$R$-triviality of Cohomology Classes and the Norm Principle}
 Let ${\bf G}$ be a semisimple algebraic group over $K$, ${\bf Z}\subset {\bf G}$ a central subgroup and let 
\begin{equation}\label{R-trivial}
[\xi]\in {\rm Ker}\,[H^1(K,{\bf Z})\longrightarrow H^1(K,{\bf G})].
\end{equation}

\smallskip

\noindent
{\bf Definition.}
 We say that $[\xi]$ is \emph{$R$-trivial} if there exists
$$
c\in {\rm Ker}\,[H^1(K(t),{\bf Z})\longrightarrow H^1(K(t),{\bf G})],
$$
with $K(t)$ a purely transcendental extension of $K$, such that $c=[\xi(t)]$ with $\xi(t)$ defined at $t=0, 1$ and satisfying $[\xi(0)]=1$ and
$[\xi(1)]=[\xi]$.

\begin{remark} The above definition requires some clarification. Here and below, if 
${\bf G}$ is an algebraic group over $K$, then an element in 
${\bf G}(K(t))$ (resp.\ a cocycle in $Z^1(K(t),{\bf G})$),  where $t$ is a variable over $K$, 
is said to be \emph{defined at 0 and 1} if it is in the image of
${\bf G}(\mathcal O)$ (resp.\ $Z^1(\mathcal O,{\bf G})$), where $\mathcal O$ is the intersection 
in $K(t)$ of the localizations $K[t]_{(t)}$ and $K[t]_{(t-1)}$.
In particular, via the evaluation maps $\varepsilon_0,\varepsilon_1:\mathcal O\to K$, we can 
evaluate such an element or a cocycle at 0 and at 1.
\end{remark}

\smallskip

\begin{example}\label{typeA_n} 
Let $D$ be a central simple algebra of degree $n$ over $K$ and set ${\bf G}={\bf SL}(1,D)$. The centre
${\bf Z}$ of ${\bf G}$ is isomorphic to $\boldsymbol{\mu}_n$, and thus $H^1(K,{\bf Z})\simeq K^\times/K^{\times n}$. 
Moreover, $H^1(K,{\bf G})\simeq K^\times /{\rm Nrd}\,(D^\times)$.
Hence
$$
{\rm Ker} [H^1(K,{\bf Z})\longrightarrow H^1(K,{\bf G})]\simeq {\rm Nrd}(D^\times)/K^{\times n},
$$
and as $D$ is the affine space ${\mathbb A}^{n^2}$, any element of the above kernel is $R$-trivial. 
\end{example}

\begin{example}\label{Pfisterforms}
Let $f$ be a Pfister form over $K$,
set ${\bf G}={\bf Spin}(f)$ and let again ${\bf Z}\subset 
{\bf G}$ be its centre.
Then all cohomology classes in (\ref{R-trivial}) are $R$-trivial; indeed, by  \cite[Proposition 7]{Merkurjev1}
the group ${\bf G}/{\bf Z}={\bf PGO}^+(f)$ is stably rational, hence $R$-trivial, and since the canonical map
$$
({\bf G}/{\bf Z})(K)\to {\rm Ker}\,[H^1(K,{\bf Z})\to H^1(K,{\bf G})]
$$ 
is surjective, 
any element in the kernel is $R$-trivial.
\end{example}
\begin{example}\label{example} 
Let $f$ be an $n$-fold Pfister form over $K$, 
and let $g$ be a  nondegenerate subform of $f\oplus \mathbb H$
of codimension $2$, where $\mathbb H$ is the hyperbolic plane. If $d$ is the determinant of $g$ then
we have a decomposition
\begin{equation}\label{decomposition}
 g\oplus a\langle 1, -d\rangle \simeq f\oplus \mathbb H
\end{equation}
for some scalar $a\in K^\times$. 
We claim that the group ${\bf PGO}^+(g)$ is $R$-trivial, or equivalently that the multiplier of any similitude with respect to
$g$ is $R$-trivial. In particular, if ${\bf G}={\bf Spin}(g)$ 
 and $\bf{Z}$ is its centre, then arguing as in the previous example, one finds that
every element in the kernel (\ref{R-trivial})
is $R$-trivial.

To compute the group of multipliers of $g$ we first recall that every multiplier with respect to $g$ is contained in 
the set $N_{K(\sqrt{d})/K}(K(\sqrt{d})^\times)$. Therefore, it follows from (\ref{decomposition})
that a multiplier $m$ with respect to $g$ is a multiplier with respect to $f$ as well, hence  is contained in the value group 
$D(f)$ of $f$. Conversely, (\ref{decomposition})
implies that every element 
$$m\in N_{K(\sqrt{d})/K}(K(\sqrt{d})^\times)\cap D(f)$$ is a multiplier of $g$. 
Let now  $U\subset K(\sqrt{d})$ be the open subvariety consisting of all elements with nonzero
norm and let
$X\subset 
U\times \mathbb{A}^{2^n}$ be the $K$-variety
consisting of the elements $(x,y)$ satisfying $N_{K(\sqrt{d})/K}(x)=f(y)$.
Consider  the map $X\to \mathbb G_m$ given by $(x,y) \to N_{K(\sqrt{d})/K}(x)$. 
Then
the group of multipliers
of $g$ is  the image of the $K$-points of $X$. Since $X$ is $K$-rational, the group of multipliers of $g$ is $R$-trivial. 
\end{example}

\smallskip

Let ${\bf G}$ be a semisimple group over $K$, and let ${\bf Z}\subset {\bf G}$ be a central subgroup.  
For any finite separable extension $L/K$ we have the restriction map
$$
res^L_K: H^1(K,{\bf Z}) \to H^1(L,{\bf Z})
$$
and the corestriction map 
$$
cor^L_K: H^1(L,{\bf Z})\to 
H^1(K,{\bf Z}).
$$

\noindent
{\bf Definition.} Let $L/K$ be a finite field extension. 
We say that the \emph{norm principle} holds for a cohomology  class 
$$
[\eta]\in {\rm Ker}\,[H^1(L, {\bf Z})\longrightarrow H^1(L,{\bf G})]
$$
if
$$
cor^L_K([\eta]) \in {\rm Ker}\,[H^1(K,{\bf Z}) \longrightarrow H^1(K,{\bf G})].  
$$
We also say that the norm principle holds for the pair $({\bf Z},{\bf G})$ if it holds for every $[\eta]\in {\rm Ker}\,[H^1(L, {\bf Z})\longrightarrow H^1(L,{\bf G})]$ 
whenever $L/K$ is a finite extension.

\begin{theorem}[P.\ Gille \cite{Gille}]\label{normprinciple} Let $K$ be a field of arbitrary characteristic
and let $L/K$ be a finite field extension. The norm principle
holds for all $R$-trivial elements $[\eta]\in {\rm Ker}\,[H^1(L, {\bf Z})\longrightarrow H^1(L,{\bf G})]$. 
 Moreover, $cor^L_K([\eta])$ is $R$-trivial.
\end{theorem}

\subsection{Conjugacy of Maximal Tori} 
Let ${\bf G}$ be an absolutely simple semisimple algebraic group over
$K$. Let ${\bf T}$ and ${\bf T}'$ be maximal tori in ${\bf G}$ defined over $K$. Since all maximal tori become conjugate upon extension to 
$K^{sep}$ there exists $g\in {\bf G}(K^{sep})$ such that ${\bf T}'=g{\bf T}g^{-1}$. Since ${\bf T}$ and ${\bf T}'$ are defined over $K$, we have
$(g^{-1})^\tau g\in N_{\bf G}({\bf T})(K^{sep})$ for all $\tau\in {\rm Gal}(K^{sep}/K)$. Thus the class of the cocycle $(\xi_\tau)=((g^{-1})^\tau g)$ 
with coefficients in $N_{\bf G}({\bf T})$ is
a cohomological obstruction for the conjugacy of ${\bf T}$ and ${\bf T}'$ over $K$. Note that since ${\bf T}'=g{\bf T}g^{-1}$,
the twisted torus $^{(\xi_\tau)}{\bf T}$ is isomorphic to ${\bf T}'$ over $K$.

Next we will show that under some
additional assumptions, one can choose $g$ such that the cocycle $(\xi_\tau)$ takes values in ${\bf T}(K^{sep})\subset N_{\bf G}({\bf T})(K^{sep})$.
Note that for such a choice of $g$ we have $^{(\xi_\tau)}{\bf T}\simeq {\bf T}$. Therefore a necessary condition for this is that ${\bf T}$ and ${\bf T}'$ be isomorphic over $K$, 
since $^{(\xi_\tau)}{\bf T}\simeq {\bf T}'$. Furthermore, our claim will hold true in ${\bf G}$ if it does in ${\bf G}/{\bf Z}$ for some central subgroup ${\bf Z}$, and it therefore suffices to consider the
adjoint case.

Assume thus that ${\bf T}\simeq {\bf T}'$ and that ${\bf G}$ is adjoint. 
Let $F/K$ be the minimal splitting field of ${\bf T}$ (and hence of ${\bf T}'$) and let
$\Gamma={\rm Gal}(F/K)$. The group $\Gamma$ acts naturally on the character lattices 
$X({\bf T})_{\ast}$ and $X({\bf T}')_{\ast}$, and these actions preserve the root systems
$\Sigma=\Sigma(G,{\bf T})$ and $\Sigma'=\Sigma(G,{\bf T}')$. Thus we have two canonical embeddings
$\rho_1:\Gamma\hookrightarrow {\rm Aut}(\Sigma)$ and  
$\rho_2:\Gamma\hookrightarrow {\rm Aut}(\Sigma')$.
Since ${\bf G}$ is adjoint, $\Sigma$ and $\Sigma'$ generate $X({\bf T})_{\ast}$ and $X({\bf T}')_{\ast}$, respectively.
Since $\Sigma$ and $\Sigma'$ are root systems of the same type
we may identify them,
which in turn
gives rise to an identification $X({\bf T})_{\ast}=X({\bf T}')_{\ast}$. After all these identifications we obtain two actions of
$\Gamma$ on each of $\Sigma$ and $X({\bf T})_{\ast}$ through $\rho_1$ and $\rho_2$. 

\begin{lemma}\label{lattice}
  Assume that there is an inner automorphism $\rho: {\rm Aut}(\Sigma)\to {\rm Aut}(\Sigma)$ such that
$\rho|_{{\rm Im}\,\rho_1}=\rho_2\circ \rho_1^{-1}$.
 Then there is a $\Gamma$-equivariant  
automorphism  $X({\bf T})_{\ast}\to X({\bf T})_{\ast}$ preserving the root system $\Sigma$,
where   $\Gamma$ acts on the domain through $\rho_1$ and on the codomain through $\rho_2$.
\end{lemma}
\begin{proof} Let $\rho={\rm Int}(a)$ where $a\in {\rm Aut}(\Sigma)$. The map $a:\Sigma\to \Sigma$ can be extended uniquely to an automorphism $a_{X({\bf T})_\ast}:X({\bf T})_\ast\to
X({\bf T})_\ast$ preserving roots. It is straightforward to check that  it is $\Gamma$-equivariant.
\end{proof}

We are now ready to conclude this section with the following theorem. Since we will mainly be 
concerned with outer forms of type $\mathrm A_2$, it is stated for groups of outer type.

\begin{theorem}\label{conjugacy} 
Let ${\bf G}$ be an absolutely simple  semisimple group over $K$ of outer type with 
$|{\rm Out}({\bf G})|=2$, and let ${\bf T}$ and ${\bf T}'$ be two isomorphic maximal tori in ${\bf G}$ defined over $K$, 
with corresponding root systems $\Sigma$ and
$\Sigma'$, respectively. Assume that there is an inner automorphism $\rho: {\rm Aut}(\Sigma)\to 
{\rm Aut}(\Sigma)$ such that
$\rho|_{{\rm Im}\,\rho_1}=\rho_2\circ \rho_1^{-1}$, where $\rho_1$ and $\rho_2$ are the above 
embeddings of $\Gamma$ into ${\rm Aut}(\Sigma)$. If there is $f\in {\rm Aut}({\bf G})(K)\setminus {\rm Int}({\bf G})(K)$ 
such that $f({\bf T})={\bf T}$, then 
there is $g\in {\bf G}(K^{sep})$ such that $g{\bf T}g^{-1}={\bf T}'$ and $(g^{-1})^\tau g\in {\bf T}(K^{sep})$ for all 
$\tau\in {\rm Gal}(K^{sep}/K)$.
\end{theorem}
\begin{proof}
 Without loss of generality we may assume that ${\bf G}$ is adjoint. 
The assumptions of Lemma \ref{lattice} are satisfied. Let thus $$a_{X({\bf T})_\ast}:X({\bf T})_\ast\to X({\bf T})_\ast$$ 
be the $\Gamma$-equivariant map constructed in that lemma. 
Using the identification of $X({\bf T})_\ast$ and $X({\bf T}')_\ast$, we obtain a  $\Gamma$-equivariant 
map $X({\bf T})_\ast\to X({\bf T}')_\ast$, which can be extended
to a $K$-isomorphism $a:{\bf T}\to {\bf T}'$ that induces an isomorphism between $\Sigma$ and $\Sigma'$.
By~\cite[Theorem$'$ 32.1]{Humphreys}
the map $a_{\bf T}$ can be further extended to an automorphism $a_{\bf G}:{\bf G} \to {\bf G}$.  Replacing  
$a_{\bf G}$ with $a_{\bf G}\circ f$, if necessary,
we may assume that $a_{\bf G}$ is inner, say $a_{\bf G}={\rm Int}(g)$ where $g\in {\bf G}(K^{sep})$. 
Since ${\rm Int}(g)|_{\bf T}: {\bf T} \to {\bf T}'$ is a $K$-isomorphism and ${\rm Int}((g^{-1})^\tau g)$
fixes $\Sigma$, it follows that $(g^{-1})^\tau g\in {\bf T}(K^{sep})$ for all 
$\tau\in {\rm Gal}(K^{sep}/K)$.
\end{proof}
\begin{example}\label{typeA_2}
We keep the above notation.  Let $E/K$ be a quadratic \'etale extension and let $B$ be a central simple algebra 
of degree $3$ over $E$ equipped with an involution $\sigma$ of the second kind. Consider
two isomorphic cubic subfields $L,L'\subset B_\sigma$ where $B_\sigma\subset B$ is the subset consisting
of $\sigma$-invariant elements. Since the maximal subfields $L\cdot E$ and $L'\cdot E$ of $B$ are $\sigma$-stable,
they give rise to two maximal $K$-tori ${\bf T}$ and ${\bf T}'$ in ${\bf G}={\bf SU}(B,\sigma)$, given by
$$
{\bf T}=\{ x\in L\cdot E \ | \ \sigma(x)x=1,\ \ {\rm Nrd}(x)=1\};
$$
$$
{\bf T}'=\{ x\in L'\cdot E \ | \ \sigma(x)x=1,\ \ {\rm Nrd}(x)=1\}.
$$   
Clearly,  ${\bf T}\simeq {\bf T}'$ and the Galois group $\Gamma$ of the minimal splitting field of ${\bf T}$ (and hence of ${\bf T}'$) 
is of order
divisible by $6$. Since ${\bf SU}(B,\sigma)$ is of type $\mathrm{A}_2$, we have
$W(\mathrm{A}_2)\simeq S_3$, and the automorphism group of its root system $\Sigma$, 
$${\rm Aut}(\Sigma)\simeq W(\mathrm{A}_2)\times \mathbb{Z}/2\simeq S_3\times \mathbb{Z}/2,
$$
is of order $12$. Thus $\Gamma$ has order $6$ or $12$.

\smallskip

\noindent
{\it Case $1$}: $|\Gamma|=12$. Then ${\rm Im}\,\rho_1$ and ${\rm Im}\,\rho_2$ coincide
with ${\rm Aut}(\Sigma)$.  Note that $$\rho_2\circ \rho_1^{-1}:{\rm Aut}(\Sigma)\to {\rm Aut}(\Sigma)$$ 
preserves the Weyl group $W(\mathrm{A}_2)$ since 
$\rho^{-1}_1(W(\mathrm{A}_2))$ (resp. $\rho_2^{-1}(W(\mathrm{A}_2))$) coincides with 
${\rm Gal}(F/E) < {\rm Gal}(F/K)$, where $F/K$
is the Galois closure of $L\cdot E/K$. 
Hence $\rho_2\circ \rho_1^{-1}$
obviously satisfies all the assumptions in Theorem~\ref{conjugacy}, and the map $f(x)=\sigma(x)^{-1}$
is an outer automorphism of ${\bf G}$ over $K$ preserving ${\bf T}$.
   
\smallskip

\noindent
{\it Case $2$}: $|\Gamma|=6$. The automorphism group ${\rm Aut}(\Sigma)$ has 3 subgroups of order
$6$, namely $\Gamma_1=W(\mathrm{A}_2)=S_3\times 0$,  the subgroup $\Gamma_2\subset S_3\times 
\mathbb{Z}/2$ 
generated by the two elements $((123),0)$ and $((12),1)$, where $(123)$ and  $(12)$ are standard
cycles in $S_3$, and the cyclic subgroup $\Gamma_3\subset S_3\times \mathbb{Z}/2$ of order $6$  
generated by the two elements
$((123),0)$ and $({\rm Id},1)$.

Since ${\bf G}$ has outer type, we know from \cite[Lemma 4.1]{PR} that $\Gamma$ does not embed into
$\Gamma_1=W(\mathrm{A}_2)\simeq S_3$.
If $\rho_1(\Gamma)= \Gamma_2\simeq S_3$, then ${\rm Im}\,\rho_1={\rm Im}\,\rho_2$, and since every automorphism of $\Gamma_2$ is obviously inner,
the automorphism $\rho_2\circ \rho_1^{-1}$ of $\Gamma_2$ can be extended to an inner automorphism of ${\rm Aut}(\Sigma)$. If instead
$\rho_1(\Gamma)= \Gamma_3\simeq \mathbb{Z}/3\times \mathbb{Z}/2$,
then again ${\rm Im}\,\rho_1={\rm Im}\,\rho_2$.
The group $\Gamma_3$ has a unique nontrivial automorphism  given by $x\mapsto x^{-1}$, and one 
easily checks that it  is the restriction of an inner automorphism of ${\rm Aut}(\Sigma)$.

Thus in all cases, the hypothesis of Theorem~\ref{conjugacy} are satisfied. 
\end{example}

\section{Subgroups of the Automorphism Group of an Albert Algebra}
In this section, we will study automorphisms of Albert algebras related to 9-dimensional subalgebras. The main result of this section
is the rationality, hence $R$-triviality, of the group of all automorphisms stabilizing a 9-dimensional subalgebra.

\subsection{$9$-dimensional Subalgebras and Their Automorphisms}\label{9-dimensional}
Let $A$ be an arbitrary  division Albert algebra over  $K$.  By~\cite[37.12 (2)]{KMRT}, 
any proper nontrivial subalgebra of $A$ is either 
a cubic field extension $K\subset L\subset A$ or a $9$-dimensional subalgebra $K\subset S \subset A$.
Furthermore, $S$ is of the form $S=D^+$ where $D$ is a central simple algebra of degree $3$ over $K$ or
$S=B_\sigma^+$
where $B$ is a central division algebra of degree $3$ over a quadratic field extension $E/K$ equipped
with an involution $\sigma$ of the second kind. For later use we record some facts related to automorphism 
groups of $D^+,B_\sigma^+$ and their extensions to automorphisms of $A$.

First, let $S=D^+$. By \cite[39.14 (2)]{KMRT}, the algebra $A$ has a presentation $A=D\oplus D\oplus D$ (as a vector space) where the subalgebra $S$ coincides with the first
component. By~\cite[37.7]{KMRT}, we have an exact sequence 
$$
1\longrightarrow {\bf Aut}(D) \longrightarrow {\bf Aut}(D^+) \longrightarrow \mathbb{Z}/2\longrightarrow 1,
$$
where the first term is the automorphism group of $D$ as an associative algebra, and this sequence is 
split  if and only if
$D$ is split.
Since $D$ is a division algebra, any $K$-automorphism of $D^+$ thus comes from ${\bf Aut}(D)(K)={\rm Aut}(D)$, 
i.e.\ is given
by conjugation $x\mapsto dxd^{-1}$ for some $d\in D^\times$. Moreover, such an automorphism can be extended to $A$
by the formula
$$
(x,y,z) \mapsto ( gxg^{-1}, gyg^{-1}, gzg^{-1}).
$$
Thus the sequence implies that $${\bf Aut}(D^+)^\circ\simeq {\bf PGL}(1,D),$$ 
where $^\circ$ denotes the identity connected component.
Note that ${\bf PGL}(1,D)$  is rational, hence in particular $R$-trivial.

Next, let $S=B_\sigma^+$.  Here the situation is completely analogous to that of $S=D^+$. Namely,
by~\cite[39.18 (2)]{KMRT}, the algebra $A$ admits the presentation $A=B_\sigma^+\oplus B$ as a vector space, with the first component a subalgebra. 
By~\cite[37.B]{KMRT},  the algebraic group  ${\bf Aut}(B^+_\sigma)$ is smooth and
$${\bf Aut}(B_\sigma^+)^\circ\simeq {\bf PGU}(B,\sigma).$$
Passing to the quadratic field extension $E/K$ we conclude that 
$$
{\bf Aut}(B^+_\sigma)(K)={\bf Aut}(B^+_\sigma)^\circ(K).
$$
Thus
any $K$-automorphism of $B_\sigma^+$ is given by conjugation $x\mapsto bxb^{-1}$ for some 
$b\in B^\times$ satisfying $b\sigma(b)\in K$. 
Since ${\bf PGU}(B,\sigma)$  has rank 
$2$ it is rational, hence $R$-trivial. In Corollary~\ref{lifting} below we shall see that any $K$-automorphism of $B_\sigma^+$ can be  extended to a $K$-automorphism of $A$.

\subsection{The Group ${\bf Aut}(A/B_\sigma^+)$}\label{Aut(A/B)}
Note that if $E=K\times K$ and $B=D\otimes_K E$ with the flip involution $\sigma$, then $B_\sigma^+$ is equal to $D^+$ embedded diagonally into $B$.
This provides a unified treatment of both kinds of 9-dimensional subalgebras. Therefore, here and everywhere below in this section, we will
let $E/K$ be a quadratic \'etale extension, including the possibility of $E$ being split; doing so, any 9-dimensional subalgebra of $A$ is of the form $B_\sigma^+$.

As above, let $A=B_\sigma^+\oplus B$.
Let $${\bf Aut}(A/B_\sigma^+)\subset {\bf Aut}(A)={\bf H}$$ be the closed subgroup consisting of the automorphisms
of $A$  fixing $B_\sigma^+$
pointwise.  One knows (see~\cite[39.B]{KMRT}) that 
the algebraic group ${\bf Aut}(A/B_\sigma^+)$ is connected simple simply connected of type  $\mathrm{A}_2$. Hence
$${\bf Aut}(A/B_\sigma^+)\simeq {\bf SU}(B,\tau)$$ 
where $\tau$ is some involution of the second kind on $B$, which in general is different from $\sigma$. 
Being an algebraic group of rank $2$ this group is rational, hence $R$-trivial.

\subsection{The Group ${\bf Aut}(A,B_\sigma^+)$}\label{secondgroup}  
Let
$$
{\bf Aut}(A,B_\sigma^+)\subset {\bf Aut}(A)={\bf H}$$ be the closed subgroup consisting of the automorphisms of $A$ stabilizing 
$B_\sigma^+$. By~\cite[Proposition 39.16]{KMRT}, we have an exact sequence
$$
1\longrightarrow {\bf Aut}(A/B_\sigma^+)\longrightarrow {\bf Aut}(A,B_\sigma^+)\longrightarrow {\bf Aut}(B_\sigma^+)
\longrightarrow 1.
$$ 
Moreover, by~\cite[39.12]{KMRT},
$$
{\bf Aut}(A,B_\sigma^+)^\circ(K)={\bf Aut}(A,B_\sigma^+)(K).
$$ 
Furthermore, the above sequence induces the exact sequence
\begin{equation}\label{connected}
1\longrightarrow {\bf Aut}(A/B_\sigma^+)\longrightarrow {\bf Aut}(A,B_\sigma^+)^{\circ}\longrightarrow {\bf Aut}(B_\sigma^+)^{\circ}
\longrightarrow 1.
\end{equation} 
Thus,
$$
{\bf Aut}(A,B_\sigma^+)^\circ/{\bf Aut}(A/B_\sigma^+)\simeq {\bf Aut}(B_\sigma^+)^{\circ}\simeq {\bf PGU}(B,\sigma)\simeq {\bf SU}(B,\sigma)/{\bf Z},
$$
where  ${\bf Z}\subset {\bf SU}(B,\sigma)$ is the
centre.

From the point of view algebraic groups, 
(\ref{connected}) implies that
the algebraic group ${\bf G}:={\bf Aut}(A,B_\sigma^+)^\circ$  is semisimple and is an almost direct product of 
two simple simply connected groups of type $\mathrm{A}_2$: the first is  
$$
{\bf G}_1:={\bf Aut}(A/B_\sigma^+)={\bf SU}(B,\tau)
$$
and the second is isomorphic to ${\bf G}_2:={\bf SU}(B,\sigma)$. The centres ${\bf Z}_1$ and ${\bf Z}_2$ of ${\bf G}_1$ and ${\bf G}_2$, respectively, are both
isomorphic to ${\bf Z}=R_{E/K}^{(1)}(\boldsymbol{\mu}_3)$, and ${\bf G}_1\cap {\bf G}_2={\bf Z}$ (see \cite[39.12]{KMRT}). 
Thus we have the exact sequence
\begin{equation}\label{sequence1}
1\longrightarrow {\bf Z} \longrightarrow {\bf G}_1\times {\bf G}_2 \stackrel{\phi}{\longrightarrow} G \longrightarrow 1,
\end{equation}
where ${\bf Z}$ is embedded \emph{codiagonally}, i.e.\ via $z\mapsto (z,z^{-1})$.
Identifying the image $\phi({\bf G}_1\times 1)\subset {\bf G}$ with ${\bf G}_1$, we recover the sequence~(\ref{connected}) in the form
\begin{equation}\label{sequence2}
1\longrightarrow {\bf G}_1 \longrightarrow  {\bf G} \stackrel{\psi}{\longrightarrow}  {\bf G}/{\bf G}_1 \longrightarrow 1.
\end{equation}     
Note that ${\bf G}/{\bf G}_1\simeq {\bf G}_2/{\bf Z}$.

\subsection{Rationality of ${\bf Aut}(A,B_\sigma^+)^\circ$}

We keep the notation introduced in Section \ref{secondgroup}. 
\begin{proposition}\label{G}
The group ${\bf G}={\bf Aut}(A,B_\sigma^+)^\circ$ is rational over $K$, hence $R$-trivial.
\end{proposition}

\begin{proof}
Consider the exact sequences ~(\ref{sequence1}) and (\ref{sequence2}).
The two groups ${\bf G}_1$ and ${\bf G}_2/{\bf Z}$ in (\ref{sequence2}), being groups of rank $2$, are rational over $K$. Therefore, it
suffices to show that $\psi$ has a rational section.
This is equivalent to proving that
$$
{\rm Ker}\,[H^1(F,{\bf G}_1) \longrightarrow H^1(F,{\bf G})]=1
$$
for all field extensions $F/K$. 

Fix a field extension $F/K$ and let $[\xi]\times 1\in H^1(F,{\bf G}_1\times 1)$ be a class whose image in $H^1(F,{\bf G})$ is trivial. 
From (\ref{sequence1}) it follows that
there is $[\lambda]\in H^1(F,{\bf Z})$ whose image in $H^1(F, {\bf G}_1\times {\bf G}_2)$ is $[\xi]\times 1$. Since ${\bf Z}$ is embedded codiagonally
into ${\bf G}_1\times {\bf G}_2$, the image of $[\lambda]$ under
the natural map $$H^1(F,{\bf Z})\to H^1(F,{\bf G}_2)$$ is trivial. We distinguish two cases.

\smallskip

\noindent
{\it Case $1$}: The quadratic  \'etale extension $E/K$ is split, i.e. $E=K\times K$. Then 
up to $K$-isomorphism we may assume that ${\bf G}_1={\bf SL}(1,D_1)$ and
${\bf G}_2={\bf SL}(1,D_2)$, where $D_1$ and $D_2$ are central simple algebras over $K$ of degree $3$, and either
$D_2=D_1$ or $D_2=D_1^{\rm op}$. In both cases their centres are $\boldsymbol{\mu}_3$, whence
$H^1(F,{\bf Z})\simeq F^\times /F^{\times 3}$, so that the class $[\lambda]\in H^1(F,{\bf Z})$ is represented by some 
$f\in F^\times$. The fact that $$H^1(F,{\bf G}_2)\simeq F^\times/ {\rm Nrd}(D_2^\times)$$ together with the image
of $[\lambda]$ in $H^1(F,{\bf G}_2)$ being trivial then imply that $f$ is a reduced norm in $D_2$,  
and hence also in $D_1$. This implies that the image $[\xi]$ of $[\lambda]$ in $H^1(F,{\bf G}_1)$ is trivial which completes the proof in this case.
\smallskip

\noindent
{\it Case $2$}: $E/K$ is a separable field extension. Let $F'=F\cdot E$, and consider the class $[\lambda]^2\in H^1(F,{\bf Z})$. From Case $1$ we know that
$res_F^{F'}([\lambda]^2)$ is contained in ${\rm Ker}\,[H^1(F',{\bf Z})\to H^1(F',{\bf G}_1)]$. It follows from 
Example~\ref{typeA_n} and the norm
principle (Theorem~\ref{normprinciple}) that 
$$
[\lambda]=[\lambda]^4=cor^{F'}_F(res_F^{F'}([\lambda]^2)
$$ 
is contained
in  ${\rm Ker}\,[H^1(F,{\bf Z})\to H^1(F,{\bf G}_1)]$, and the proof is complete. 
\end{proof}

As a direct consequence of our proof we have the following.
\begin{corollary}\label{lifting}  
For any field extension $F/K$, the canonical map ${\bf G}(F) \to ({\bf G}_2/{\bf Z})(F)$
is surjective. In other words, if $A_F$ is a division Albert algebra, then
$${\bf Aut}(A,B_\sigma^+)(F) \longrightarrow {\bf Aut}(B_\sigma^+)(F)$$
is surjective.
\end{corollary}

\section{Subgroups of the Structure Group of an Albert Algebra}
We now turn to the structure group and consider, along the same lines as in the previous section, subgroups of it related to 9-dimensional subalgebras. We continue using 
the convention
that the quadratic \'etale algebra $E$ involved in the definition of $B_\sigma^+$ may be split.

\subsection{The Group ${\bf Str}(B_\sigma^+)$}\label{StrB}
The structure group ${\bf Str}(B_\sigma^+)$ of the Jordan algebra $B_\sigma^+$
is the subgroup of ${\bf GL}(B_\sigma^+)$ consisting of all similitudes. More precisely, for any $K$-ring $R$,
$$
{\bf Str}(B_\sigma^+)(R) = \{x \in {\bf GL}(B_\sigma^+)(R) \ | \  {\rm Nrd}(x(b)) = \nu(x){\rm Nrd}(b) \ \   \forall b\in B_\sigma^+\otimes_K R\}
$$
where the \textit{multiplier} $\nu(x)\in R^\times$ depends only on $x$. By \cite[Chap. V, Thm. 5.12.10]{Jacobson}, the group ${\rm Str}(B_\sigma^+)={\bf Str}(B_\sigma^+)(K)$ consists of the linear maps of the form
$$
B_\sigma^+\longrightarrow B_\sigma^+,\ \  x \mapsto \lambda bx\sigma(b),
$$
where $b\in B^\times$ and $\lambda\in K^\times$. It follows that  
 we have a surjective
map of algebraic groups 
$$
{\bf G}_m \times R_{E/K}({\bf GL}(1,B)) \longrightarrow {\bf Str}(B_\sigma^+)^\circ, $$ 
 and that 
${\bf Str}(B_\sigma^+)^\circ(K)={\bf Str}(B_\sigma^+)(K)$. The kernel of this map is the torus $R_{E/K}({\bf G}_{m,E})$, where the embedding 
 $$
 R_{E/K}({\bf G}_{m,E})\hookrightarrow {\bf G}_m \times R_{E/K}({\bf GL}(1,B))
 $$
 is given by $x\mapsto (N_{E/K}(x^{-1}),x)$. Thus, we have the exact sequence
\begin{equation}\label{sequence3}
1\longrightarrow R_{E/K}({\bf G}_{m,E}) \longrightarrow 
{\bf G}_m \times R_{E/K}({\bf GL}(1,B)) \stackrel{\phi}{\longrightarrow} {\bf Str}(B_\sigma^+)^\circ
\longrightarrow 1.
\end{equation}
\begin{lemma}\label{Str(B)} 
The group ${\bf Str}(B_\sigma^+)^\circ$ is rational, hence in particular $R$-trivial.
\end{lemma}
\begin{proof} We identify $\phi({\bf G}_m\times 1)\subset {\bf Str}(B_\sigma^+)^\circ$ with ${\bf G}_m$.
Since for any field extension $F/K$ one has $H^1(F,{\bf G}_m)=1$,
  the canonical map
$${\bf Str}(B_\sigma^+)^\circ\to {\bf Str}(B_\sigma^+)^\circ/{\bf G}_m$$ has a rational section. Therefore, it suffices
to establish the rationality of ${\bf Str}(B_\sigma^+)^\circ/{\bf G}_m$. By (\ref{sequence3})
we have
$$
{\bf Str}(B_\sigma^+)^\circ/{\bf G}_m\simeq R_{E/K}({\bf GL}(1,B^\times))/R_{E/K}({\bf G}_{m,E}),
$$
which is clearly rational.
\end{proof}

\subsection{The Group ${\bf Str}(A,B_\sigma^+)$}\label{Str(A,B)} 
Let ${\bf Str}(A,B_\sigma^+)$ (resp.\ ${\bf Str}(A/B_\sigma^+)$)  be the closed subgroup of ${\bf Str}(A)$
consisting of those elements stabilizing $B_\sigma^+$ (resp.\ fixing $B_\sigma^+$ pointwise). Note that since the elements of ${\bf Str}(A/B_\sigma^+)$
fix the identity element of $A$, it follows from \cite[5.9.4]{SV} that 
${\bf Str}(A/B_\sigma^+)={\bf Aut}(A/B_\sigma^+)$. This group is the kernel of the canonical restriction map
$${\bf Str}(A,B_\sigma^+)^\circ \to {\bf Str}(B_\sigma^+)^\circ.$$
Since, by \cite[Proposition 7.2.4]{GP}, this  map is surjective on the level of $K^{sep}$-points, we have the exact sequence 
\begin{equation}\label{firstexact}
 1\longrightarrow {\bf Aut}(A/B_\sigma^+) \longrightarrow {\bf Str}(A,B_\sigma^+)^\circ \stackrel{\phi}{\longrightarrow}
{\bf Str}(B_\sigma^+)^\circ \longrightarrow 1
\end{equation}
of algebraic groups.

\begin{proposition}\label{structuralgroup}
 The group ${\bf Str}(A,B_\sigma^+)^\circ$ is rational, hence
$R$-trivial.
\end{proposition}
\begin{proof} By \cite{GP}, for any field extension $F/K$ the map
$$
{\bf Str}(A,B_\sigma^+)^\circ(F) \stackrel{\phi_F}{\longrightarrow}
{\bf Str}(B_\sigma^+)^\circ(F)
$$
 is surjective, implying that $\phi$ has a rational section. Thus ${\bf Str}(A,B_\sigma^+)^\circ$
is birationally isomorphic to ${\bf Aut}(A/B_\sigma^+) \times {\bf Str}(B_\sigma^+)^\circ$.
It remains to note that 
being a group of rank $2$ the group  ${\bf Aut}(A/B_\sigma^+)$ is rational, and by Lemma~\ref{Str(B)}
the group ${\bf Str}(B_\sigma^+)^\circ$ is rational. 
\end{proof}

\begin{corollary}\label{surjectivity}
The natural map ${\rm Str}(A,B_\sigma^+)\to {\rm Str}(B_\sigma^+)$ is surjective.
\end{corollary}
\begin{proof}
From the proof of Proposition~\ref{structuralgroup} it follows that the map $${\bf Str}(A,B_\sigma^+)^\circ(K)\to {\bf Str}(B_\sigma^+)^\circ(K)$$ is surjective, and from~\ref{StrB} 
we know that ${\bf Str}(B_\sigma^+)(K)={\bf Str}(B_\sigma^+)^\circ(K)$. The assertion follows.
\end{proof}

\section{The Weak Skolem--Noether Property  for Isomorphic Embeddings}\label{SN}

Let $A$ be an Albert algebra over a field $K$. Let $K\subset L\subset A$ and
$K\subset L'\subset A$ be two isomorphic separable cubic field extensions; one says 
that they are \emph{weakly equivalent} if there exists an element $g\in {\rm Str}(A)$ such that
$g(L)=L'$. In their joint paper \cite{GP}, S.\ Garibaldi and H.\ Petersson prove that this 
property always 
holds, and following their terminology we call  
it the \emph{weak Skolem--Noether property for isomorphic embeddings}. The goal of this section 
is to give an alternative proof of their result which is much shorter and more conceptual. It is
based on the technique of conjugacy of maximal tori, detailed in Section \ref{preliminaries}.

We start with the intermediate step of a $9$-dimensional Jordan algebra $B^+_\sigma$.

\begin{proposition} Let $L$ and $L'$ be two isomorphic separable cubic field extensions of the base field $K$ contained in the subalgebra $B_\sigma^+$. 
Then there exists an element $s\in {\rm Str}(B_\sigma^+)$ such that $s(L)=L'$.
 \end{proposition}
 \begin{proof}
 Let $E/K$ be the 
 \'etale quadratic extension over which $B$ is defined. 
The two cubic fields $L$ and $L'$ give rise to the two $4$-dimensional (maximal) tori 
$$
{\bf T}=\{ x\in R_{L\cdot E/K}({\bf G}_{m,L\cdot E})\ | \  \sigma(x)x\in {\bf G}_{m,K}\}
$$
and
$$
{\bf T}'=\{ x\in R_{L'\cdot E/K}({\bf G}_{m,L'\cdot E})\ | \ \sigma(x)x\in {\bf G}_{m,K}\}.
$$
over $K$ in ${\bf Sim}(B,\sigma)$. From the point of view of algebraic groups the group ${\bf Sim}(B,\sigma)$ is a reductive group
which is the almost direct product of the central $2$-dimensional torus  ${\bf P}=R_{E/K}({\bf G}_{m,E})$
and the simple simply connected  group ${\bf SU}(B,\sigma)$  of outer type $\mathrm{A}_2$. Note that ${\bf P}$ 
is contained
in both ${\bf T}$ and ${\bf T}'$. Let ${\bf T}_{ss}={\bf T}\cap  {\bf SU}(B,\sigma)$ and ${\bf T}'_{ss}={\bf T}'\cap
{\bf SU}(B,\sigma)$. Then ${\bf T}={\bf P}\cdot {\bf T}_{ss}$ and ${\bf T}'={\bf P}\cdot {\bf T}'_{ss}$.

By Theorem~\ref{conjugacy} and Example~\ref{typeA_2}, there exists $b\in {\bf SU}(B,\sigma)(K^{sep})$
such that $b{\bf T}_{ss}b^{-1}={\bf T}'_{ss}$ and $b^{-\tau+1}
   \in {\bf T}_{ss}(K^{sep})$ for all $\tau\in {\rm Gal}(K^{sep}/K)$. Since ${\bf P}\subset {\bf Sim}(B,\sigma)$ 
is a central torus  we also have
$b{\bf T}b^{-1}={\bf T}'$. 

Consider the following commutative diagram with exact rows:
 \[
\begin{CD}
1  @>>> R_{E/K}({\bf G}_{m,E}) @>>> {\bf Sim}(B,\sigma) @>{\phi_1}>> {\bf Aut}(B_\sigma^+)^\circ  @>>> 1     \\
@. @VV{\mathrm{Id}}V    @VV{\lambda_2}V  @VVV     \\
1  @>>> R_{E/K}({\bf G}_{m,E}) @>>> {\bf J}  @>{\phi_2}>> 
{\bf Str}(B_\sigma^+)^\circ  @>>> 1.   \\
\end{CD}
\]
Here ${\bf J}={\bf G}_{m,K}\times R_{E/K}({\bf GL}(1,B))$, so that the lower sequence is the exact sequence \eqref{sequence3}, $\phi_1$ is given by conjugation, 
$$
\phi_2(x,y): B_\sigma^+\longrightarrow B_\sigma^+,\ \ a\mapsto xya\sigma(y),
$$
and
$$
\lambda_2(x)=(\nu(x^{-1}),x)=(x^{-1}x^{-\sigma},x)$$
where $\nu(x)$ is the multiplier of $x$.
The map $\lambda_2$  takes ${\bf T}$ and ${\bf T}'$ into the quasi-trivial 
tori
$$
\widetilde{{\bf T}}={\bf G}_{m,K}\times R_{L\cdot E/K}({\bf G}_{m,L\cdot E})
$$
and
$$
\widetilde{{\bf T}}'={\bf G}_{m,K}\times R_{L'\cdot E/K}({\bf G}_{m,L'\cdot E}),
$$
respectively. It is easy to see that $\widetilde{{\bf T}}$ (resp.\ $\widetilde{{\bf T}'}$) is the centralizer of ${\bf T}$ (resp.\ ${\bf T}'$)
in ${\bf G}_{m,K}\times R_{E/K}({\bf GL}(1,B))$. Therefore the equality
$b{\bf T}b^{-1}={\bf T}'$ implies $b\widetilde{{\bf T}}b^{-1}=\widetilde{{\bf T}}'$ and
$$
\lambda_2(b^{-\tau+1})=(\lambda_2(b))^{-\tau+1}=(\nu(b^{-1})^{-\tau+1},b^{-\tau+1})
\in \widetilde{{\bf T}}(K^{sep}).$$
Since $H^1(K,R_{L\cdot E/K}({\bf G}_{m,L\cdot E}))=1$ we can pick an element
$$
a\in R_{L\cdot E/K}({\bf G}_{m,L\cdot E})(K^{sep})=((L\cdot E)\otimes_K K^{sep})^\times
$$ 
such that
$b^{-\tau+1}=a^{-\tau+1}$. Clearly, $c:=ba^{-1}$ is ${\rm Gal}(K^{sep}/K)$-invariant, hence
$c\in B^\times$ and, defining $s$ as the map $x\mapsto cx\sigma(c)$, the following claim completes
the proof of the proposition.

\smallskip

\noindent
{\it Claim}: $cL\sigma(c)=L'$.

\smallskip

\noindent
The claim is equivalent to $c(L\cdot E)\sigma(c)=L'\cdot E$. Moreover, it suffices to show that 
$$c((L\cdot E)\otimes_K K^{sep})\sigma(c)=(L'\cdot E)\otimes_K K^{sep}.$$
But $\sigma(c)=\sigma(a^{-1})\sigma(b)$ and both $a^{-1}$ and $\sigma(a^{-1})$ are in 
$(L\cdot E)\otimes_K K^{sep}$. Therefore it suffices to show that
$$b((L\cdot E)\otimes_K K^{sep})\sigma(b)=(L'\cdot E)\otimes_K K^{sep}.$$
Recall that by construction, $b$ is a similitude, hence $\sigma(b)=\nu(b)b^{-1}$, where $\nu(b)$ is the multiplier of $b$.
Since $b{\bf T}b^{-1}={\bf T}'$, the claim follows upon noting that the centralizer of ${\bf T}$ (resp.\ ${\bf T}'$) in $B\otimes_K K^{sep}$
is $(L\cdot E)\otimes_K K^{sep}$ (resp.\ $(L'\cdot E)\otimes_K K^{sep}$).
\end{proof}

The proposition in conjunction 
with Corollary~\ref{surjectivity} yields the following.

\begin{theorem}\label{comingback} 
Let $A$ be an Albert algebra over $K$ and let
$L\subset A$ and $L'\subset A$ be isomorphic separable cubic field extensions of $K$. Then 
there exists a 9-dimensional subalgebra $B_\sigma^+$ of $A$ and an element 
$s\in {\rm Str}(A,B_\sigma^+)$ such that $s(L)=L'=f(L)$.
\end{theorem}
\begin{proof}
If $L'=L$,  we can take any any $9$-dimensional subalgebra $B^+_\sigma$ containing
$L$ and choose 
$s=\mathrm{Id}$. Otherwise, the cubic subfields $L,L'$ generate a subalgebra in $A$ of the form 
$B_\sigma^+$ and we can take any lift of the element $s$ constructed in the above proposition.
\end{proof}

\section{Reduction to $\mathrm{F}_4$}\label{reductionF_4}

In this section
we will show that an arbitrary element in ${\rm Str}(A)$ can be written as a product of $R$-trivial elements  and elements
in ${\bf H}(K)$, thereby reducing the problem to subgroups of type $\mathrm{F}_4$. To begin with, we recall from \cite{ACP} how to associate, to any $a\in A^\times$, a subgroup of type $\mathrm{D}_4$ in ${\bf G}$ and a $2$-dimensional torus.
Let $L\subset A$ be the $K$-subalgebra generated by $a$, if $a$ is not a scalar multiple of the identity of $A$, and by any element that is not such a scalar multiple, if $a$ is.
Since $A$ is a division algebra, it is known that $L$ is a cubic subfield. Let again ${\bf G}=[{\bf Str}(A),{\bf Str}(A)]$ and let
$$
{\bf G}^L=\{ x\in {\bf G} \ | \ x(l)=l \ \ \forall l\in L\}.
$$
Since ${\bf G}^L$ stabilizes $1\in L\subset A$, we have ${\bf G}^L\subset {\bf H}\subset {\bf G}$.
It is known that over a separable closure of $K$ the group ${\bf G}^L$ is conjugate
to the standard subgroup in ${\bf G}$ of type $\mathrm{D}_4$ generated by roots $\alpha_2,\alpha_3,\alpha_4,
\alpha_5$. Therefore the centralizer ${\bf S}'^L=C_{\bf G}({\bf G}^L)$ of ${\bf G}^L$ in ${\bf G}$ is a $2$-dimensional torus over $K$. 
Using an explicit reduced model of $A$ over a separable closure of $K$ 
one can easily verify that $${\bf Z}^L:={\bf S}'^L\cap {\bf H}={\bf S}'^L\cap 
{\bf G}^L$$ is the centre of ${\bf G}^L$. 

Let ${\bf S}^L\subset {\bf Str}(A)$ be the $3$-dimensional torus in ${\bf Str}(A)$ generated by ${\bf S}'^L$ and ${\bf G}_m$, where the latter is embedded in ${\bf Str}(A)$
as the subgroup of homotheties. The following was proved in \cite{ACP}.

\begin{lemma}\label{LACP} With the notation above, ${\bf S}^L\simeq R_{L/K}({\bf G}_m)$ and ${\bf S}'^L\simeq R_{L/K}^{(1)}({\bf G}_m)$. Moreover,
$$\{ a\in A \ | \ x(a)=a\ \ \forall x\in {\bf G}^L(K)\}=L,$$
and the natural action of ${\bf Str}(A)$ on $A$ induces an action of ${\bf S}^L$ on $L$, which is transitive on the level of the $K^{sep}$-points  of the open subset 
$$
L^\times=\{x\in L\ | \ N(x)\not=0\}
$$ 
 of $A$ and gives rise to an exact sequence
\[1\longrightarrow {\bf Z}^L \longrightarrow {\bf S}^L \longrightarrow L^\times \longrightarrow 1.\]
\end{lemma}

From the exact sequence in cohomology associated to the sequence in the lemma we have a map $L^\times \to H^1(K,{\bf Z}^L)$, which is surjective since $H^1(K,{\bf S}^L)$ is trivial.
Using this map we can attach, to any $a\in L^\times$, a class $[\xi_a]=(a_\tau)\in H^1(K,{\bf Z}^L)$. From \cite{ACP} we moreover know that the image of this class in $H^1(K,{\bf H})$ is trivial
 if $a$ is in the ${\rm Str}(A)$-orbit of $1$.

\smallskip

Let now $g\in {\rm Str}(A)$ be an arbitrary element and set 
$a:=g(1)$. As detailed above we attach to $a$ a subfield $L$ in $A$ and with it
the closed subgroups ${\bf G}^L$, ${\bf Z}^L$ and ${\bf S}^L$ of ${\bf Str}(A)$, as well as the class $[\xi_a]=(a_\tau)\in H^1(K,{\bf Z}^L)$. Since the image of this class
in $H^1(K,{\bf H})$ is trivial, there exists $f\in {\bf H}(K^{sep})$  such that
$a_{\tau}=(f^{-1})^\tau f$ for all $\tau\in {\rm Gal}(K^{sep}/K)$.
\begin{lemma} 
The subset $f(L)\subset A_{K^{sep}}=A\otimes_K K^{sep}$ is contained in $A=A\otimes_K K$.
\end{lemma}
\begin{proof} Let $l\in L$. We need to show that $f(l)$ is ${\rm Gal}(K^{sep}/K)$-invariant.
Take any $\tau\in {\rm Gal}(K^{sep}/K)$. Then
$$
\tau(f(l))=\tau(f)(\tau(l))=\tau(f)(l)=f (f^{-1}f^\tau(l))=f a^{-1}_{\tau}(l)=f(l),
$$
since ${\bf Z}^L\subset {\bf G}^L$ fixes $L$ pointwise, and the statement follows.
\end{proof}
Since ${\bf H}$ is the automorphism group of $A$, the above lemma implies that the map
$L\to L'=f(L)$ given by $l \mapsto f(l)$ is a field isomorphism over $K$. Assume that $L\not=L'$.
Let $B_\sigma^+\subset A$ be
the $9$-dimensional subalgebra generated by $L$ and $L'$.
Recall from \cite{ACP} that $\xi_a$ is constructed explicitly as follows.
Choose $t\in {\bf S}^L(K^{sep})$ such that $t(1)=a=g(1)$; this is possible by Lemma \ref{LACP}.
Then $a_\tau=t^{-\tau+1}$, from which we conclude that $ft^{-1}$ is defined over $K$ and that
$$
ft^{-1}(g(1))=ft^{-1}(a)=f(1)=1,
$$
which implies that $ft^{-1}g\in {\bf H}(K)$. Thus, modulo $K$-points of ${\bf H}$, we may assume 
that $g=ft^{-1}$.

Let now $s\in {\rm Str}(A,B_\sigma^+)$ be the element constructed in Theorem~\ref{comingback}. 
It is $R$-trivial, since so is ${\bf Str}(A,B_\sigma^+)^\circ$, and it satisfies $s(L)=f(L)$. Furthermore,
$$
L=s^{-1}(f(L))=s^{-1}ft^{-1}(L)=s^{-1}g(L),
$$
since $t(L)=L$. It follows that modulo $R$-trivial elements we may assume that $g(L)=L$,
i.e. $g\in {\bf Str}(A,L)(K)$, where ${\bf Str}(A,L)\subset {\bf Str}(A)$ is the subgroup of all elements stabilizing $L$.

Passing to  a separable closure of $K$ one can easily check that the connected
component of ${\bf Str}(A,L)$ is ${\bf S}^L\cdot {\bf G}^L$. Hence
$$
{\bf Str}(A,L)=N_{{\bf Str}(A)}({\bf S}^L\cdot {\bf G}^L)
$$
and
$$
{\bf Str}(A,L)/{\bf S}^L\cdot {\bf G}^L\simeq 
N_{{\bf H}}({\bf G}^L)/{\bf G}^L
\simeq {\bf Out}({\bf G}^L).
$$
By Lemma~\ref{twistingD_4} we can, if necessary, multiply $g$ by an element from $N_{\bf H}({\bf G}^L)(K)$
to obtain an element $g'\in({\bf S}^L\cdot {\bf G}^L)(K)$. To complete our reduction to subgroups of type $\mathrm{F}_4$, it remains to show that $g'$ is $R$-trivial modulo elements from ${\bf H}(K)$. This is the content of the following result.

\begin{proposition}\label{lastreduction}
Let $g\in ({\bf S}^L\cdot {\bf G}^L)(K)$. Then there exists
an $R$-trivial element $j$  in $({\bf S}^L\cdot {\bf G}^L)(K)$ 
such that $gj\in {\bf G}^L(K)\subset {\bf H}(K)$.
\end{proposition}
\begin{proof}
Our argument is based on the consideration of the exact sequence
\begin{equation}\label{base}
1\longrightarrow {\bf Z}^L\longrightarrow {\bf S}^L\times {\bf G}^L
\longrightarrow {\bf S}^L\cdot {\bf G}^L \longrightarrow 1.
\end{equation}
In the corresponding exact sequence in cohomology, the element $g$ is mapped to a class $[\eta]$ in $H^1(K,{\bf Z}^L)$. Since $H^1(K,{\bf S}^L)$ is trivial, this class
belong to
$$
{\rm Ker}\,[H^1(K,{\bf Z}^L)\longrightarrow H^1(K,{\bf G}^L)].
$$
We first prove that it is $R$-trivial. Since ${\bf Z}^L$ is a group of exponent $2$, by the norm principle it suffices to prove that $[\eta]$ becomes 
$R$-trivial after extending scalars to $L/K$. Two cases are possible.

If $L/K$ is a Galois extension, then  ${\bf G}^L_L$, being a strongly inner form of type 
$^1\mathrm{D}_4$, is of the form
${\bf G}^L_L\simeq {\bf Spin}(h)$,
where $h$ is a $3$-fold Pfister form. Hence
by Example~\ref{Pfisterforms},
the class ${res}^L_K([\eta])$ is $R$-trivial. 

If
$L/K$ is not a Galois extension, then ${\bf G}^L_L\simeq {\bf Spin}(h)$ for some $h$ satisfying
all conditions in Example~\ref{example}. Therefore ${res}^L_K([\eta])$ is also $R$-trivial.

\smallskip

Now the sequence~(\ref{base}) induces the exact sequence
$$
({\bf S}^L\times {\bf G}^L)(K(x))
\longrightarrow 
({\bf S}^L \cdot {\bf G}^L)(K(x))\longrightarrow
H^1(K(x),{\bf Z}^L)\longrightarrow 
 H^1(K(x),{\bf G}^L),
$$
where $x$ is a variable over $K$.
Let 
$$
\eta(x)\in {\rm Ker}\,[H^1(K(x),{\bf Z}^L) \longrightarrow H^1(K(x),{\bf G}^L)]
$$
be defined at 0 and 1 and such that $\eta(0)=1$ and $\eta(1)=\eta$; such an element exists since $\eta$ is $R$-trivial. Take any element 
$g(x)\in ({\bf S}^L \cdot {\bf G}^L)(K(x))$ that is defined at 0 and 1 and whose image in $H^1(K(x), {\bf Z}^L)$ 
is $\eta(x)$. Note that $g(0)^{-1}g(1)$ is $R$-trivial in 
${\bf S}^L \cdot {\bf G}^L$.
By construction,  $g(0)=u_0$ and  $g(1)=u_1g$ for some
$u_0,u_1\in ({\bf S}^L\times {\bf G}^L)(K)$. Hence
$$
g=u_1^{-1}g(1)=u_1^{-1}g(0)\left(g(0)^{-1}g(1)\right)=u_1^{-1}u_0\left(g(1)g(0)^{-1}\right).
$$
Writing $u_1^{-1}u_0=us$ for some  $u\in{\bf G}^L(K)$ and  $s\in{\bf S}^L(K)$, and noting that ${\bf S}^L$ is a rational torus, 
the proof is complete upon setting $j=\left(sg(1)g(0)^{-1}\right)^{-1}$.
\end{proof}

We have thus altogether proved the following.
\begin{theorem}\label{stabilizer}
Let $L\subset A$ be a cubic subfield and let $g\in {\rm Str}(A)$ be such that $g(L)=L$. Then modulo $R$-trivial elements,
$g$ can be written as a product $g=g_1g_2$ where $g_1(L)=g_2(L)=L$,
$g_1\in N_{\bf H}({\bf G}^L)(K)$, and $g_2\in {\bf G}^L(K)$.
\end{theorem}

\section{The End of the Proof}

We now finish the proof that ${\bf Str}(A)$ is $R$-trivial. To begin with, we recall the following fact about automorphisms of Albert algebras, including its
easy proof for convenience.

\begin{lemma}\label{cubicsubfield} Let $A$ be a division Albert algebra and let $g\in {\bf H}(K)$. Then $g$ fixes a cubic subfield in $A$ pointwise.
\end{lemma}
\begin{proof}
Since ${\bf H}$ is $K$-anisotropic, $g$ is semisimple, hence contained
in a maximal $K$-torus ${\bf T}\subset {\bf H}$. Using an explicit reduced model 
of the split Albert algebra $A\otimes_K K^{sep}$, one can easily check that over $K^{sep}$, every element 
in ${\bf T}$ fixes a commutative subalgebra 
$$W\simeq K^{sep}\times K^{\sep}\times K^{sep}$$
of $ A\otimes_K K^{sep}$ pointwise. It follows that $g$ fixes a vector $v\in A$ which is not a scalar multiple of the identity, whence
 it fixes the cubic subfield $K(v)\subset A$ generated by $v$.
\end{proof}

In order to conclude, we need one final ingredient. Let $A$ be an Albert algebra over $K$ and let $p\in A^\times$. Recall that this is equivalent to $N(p)\neq 0$, where $N$ is the
cubic norm of $A$. The \emph{isotope} $A^{(p)}$ of $A$ is the algebra with underlying linear space $A$, and multiplication
\[x\bullet_p y=x(yp)+y(xp)-(xy)p,\]
where juxtaposition denotes the multiplication of $A$. It is known that $A^{(p)}$ is an Albert algebra, and that the cubic norm of $A^{(p)}$ is $\nu N$, where $\nu=N(p)$. From this it follows
that $A^{(p)}$ is a division algebra if and only if $A$ is, and that ${\bf Str}(A)\simeq {\bf Str}(A^{(p)})$.

\begin{proposition}\cite{Thakur3} Let $A$ be a division Albert algebra over a field $K$. 
Then there exists $p\in A^\times$ such that $A^{(p)}$ contains a cubic cyclic subfield.
\end{proposition}

\begin{remark} 
If the ground field $F$ contains a cubic root of unity, then due to
a result of H.\ Petersson and M.\ Racine \cite{PetR}, one can take $p=1$.
\end{remark}

We are now ready to prove our main result.

\begin{theorem}\label{main} Let $A$ be a division Albert algebra over $K$.
Then  ${\bf Str}(A)$ is an $R$-trivial group.
\end{theorem}

\begin{proof}
By the above discussion, we may replace $A$ by any isotope $A^{(p)}$. Thus by the preceding proposition we may assume that $A$ contains a cubic cyclic subfield $L$.
Let $g\in {\rm Str}(A)$. By the results of Section \ref{reductionF_4}, summarized in Theorem \ref{stabilizer},  we may assume that
$g\in {\bf H}(K)$. Three cases are possible.
\smallskip

\noindent
{\it Case $1$}: $g$ fixes $L$, but not pointwise.  Let $F\subset A$ be a cubic field extension 
of $K$ pointwise fixed by $g$; such a field exists by Lemma \ref{cubicsubfield}.
Since $F\not=L$, $L$ and $F$ generate a 9-dimensional subalgebra $B_\sigma^+\subset A$ that is stabilized by $g$. 
In the course of the proof of Corollary~\ref{surjectivity}
we saw that $${\bf Str}(A,B_\sigma^+)(K)={\bf Str}(A,B_\sigma^+)^\circ(K).$$
By Proposition~\ref{structuralgroup}, the latter group is $R$-trivial, and therefore $g$ is $R$-trivial.

\smallskip

\noindent
{\it Case $2$}:
 $g$ fixes $L$ pointwise. Hence $g$ is in the group
$$
{\bf Str}(A/L)=\{x\in {\bf Str}(A)\ | \ x(l)=l \ \ \mbox{for all}\ l\in L\}={\bf G}^L.
$$
Since $L$ is cyclic, by Lemma~\ref{twistingD_4} there is an element $h_1\in N_{\bf H}({\bf G}^L)(K)$
that stabilizes $L$, but does not fix it pointwise. The same is true for $h_2:=h_1^{-1}g$, since $g$ fixes $L$ pointwise. 
From Case (1) we know that $h_1$ and $h_2$ are $R$-trivial, and hence so is $g$.

\smallskip

\noindent
{\it Case 3}: $g(L)\not=L$.  Since $g\in {\rm Aut}(A)$, the field $g(L)$ is a cubic cyclic subfield 
of $A$ isomorphic to $L$.  The subfields $L$ and $g(L)$ generate a 9-dimensional subalgebra $B_\sigma^+$ of $A$.
 By Theorem~\ref{comingback} there exists $h_1\in
 {\rm Str}(A,B_\sigma^+)$ such that $h_1(g(L))=L$, and by Proposition~\ref{structuralgroup}, $h_1$ is $R$-trivial. Let
 $h_2:=h_1g$. By construction it belongs to
 $$
 {\bf Str}(A,L)(K)=\{ x\in {\rm Str}(A)\ | \ x(L)=L\}.
 $$
 By Theorem~\ref{stabilizer}, $h_2$ can be written, modulo $R$-trivial elements, as a product 
 $h_2=h_3h_4$ with $h_3\in {\bf G}^L(K)$ and $h_4\in N_{\bf H}({\bf G}^L)(K)$.
 In particular $h_3$ and $h_4$ are in ${\bf H}(K)$ and stabilize $L$. By Cases (1) and (2),  the elements $h_3$ and $h_4$
 are $R$-trivial, and hence so is $h_2$ and $g$. This completes the proof.
\end{proof}

\begin{theorem}\label{adjointcase} The group $\overline{\bf G}$ is $R$-trivial.
\end{theorem}
\begin{proof} Indeed, the canonical map
${\bf Str}(A) \to \overline{\bf G}$ has a rational section, since its kernel  ${\bf G}_m$
has trivial Galois cohomology in dimension 1. We conclude with the above theorem.
\end{proof}

\section{Applications}

\subsection{The Kneser--Tits Problem for $\mathrm{E}_{7,1}^{78}$ and $\mathrm{E}_{8,2}^{78}$}

\begin{theorem} Let ${\bf G}$ be an algebraic group of type 
$\mathrm{E}_{7,1}^{78}$ or $\mathrm{E}_{8,2}^{78}$ defined over a field $K$ of an 
arbitrary characteristic. Then the Kneser--Tits conjecture for ${\bf G}$ holds.
\end{theorem}
\begin{proof} According to \cite[Remarque 7.4]{gille} we may, without loss of generality, assume that ${\rm char}(K)=0$. The Tits index of ${\bf G}$ is of the form
$$
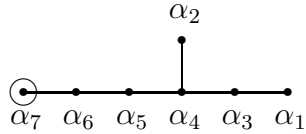
\begin{array}{ll}
\begin{picture}(125,55)
\put(00,02){\line(1,0){100}}
\put(60,02){\line(0,1){20}}
\put(00,02){\circle*{3}}
\put(20,02){\circle*{3}}
\put(40,02){\circle*{3}}
\put(60,02){\circle*{3}}
\put(80,02){\circle*{3}}
\put(100,02){\circle*{3}}
\put(60,22){\circle*{3}}
\put(00,02){\circle{10}}
\put(-5,-10){$\alpha_7$}
\put(15,-10){$\alpha_6$}
\put(35,-10){$\alpha_5$}
\put(55,-10){$\alpha_4$}
\put(75,-10){$\alpha_3$}
\put(95,-10){$\alpha_1$}
\put(55,30){$\alpha_2$}
\end{picture}\\
\end{array}
$$
or
$$
\begin{array}{ll}
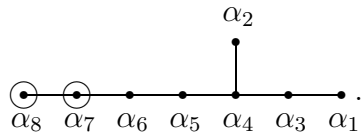
\begin{picture}(125,55)
\put(00,02){\line(1,0){120}}
\put(80,02){\line(0,1){20}}
\put(00,02){\circle*{3}}
\put(20,02){\circle*{3}}
\put(40,02){\circle*{3}}
\put(60,02){\circle*{3}}
\put(80,02){\circle*{3}}
\put(100,02){\circle*{3}}
\put(120,02){\circle*{3}}
\put(80,22){\circle*{3}}
\put(00,02){\circle{10}}
\put(20,02){\circle{10}}
\put(-5,-10){$\alpha_8$}
\put(15,-10){$\alpha_7$}
\put(35,-10){$\alpha_6$}
\put(55,-10){$\alpha_5$}
\put(75,-10){$\alpha_4$}
\put(95,-10){$\alpha_3$}
\put(115,-10){$\alpha_1$}
\put(75,30){$\alpha_2$}
\end{picture}.
\end{array}
$$

\bigskip

\medskip

\noindent
In both cases the semisimple anisotropic kernel ${\bf H}$ of ${\bf G}$ is a strongly
inner form of type $\mathrm{E}_6$. If ${\bf S}\subset {\bf G}$ is a maximal split torus 
whose centralizer is ${\bf H}$, then arguing as in \cite{ChPl} 
one easily verifies that over an algebraic closure
of $K$ the intersection ${\bf S}\cap {\bf H}$ is the centre of ${\bf H}$. 

Furthermore, by \cite[Th\'eor\`eme 7.2]{gille},
the Kneser--Tits problem has an affirmative answer if and only if
${\bf G}(K)/R=1$. It follows from the Bruhat-Tits decomposition that
$$
{\bf G}(K)/R\simeq C_{\bf S}({\bf G})/R \simeq ({\bf H}/(C_{\bf S}({\bf G})\cap {\bf H}))/R
\simeq \overline{{\bf H}}/R
$$
where $\overline{{\bf H}}$ is the corresponding adjoint group. It remains to note
that by Theorem~\ref{adjointcase}, the group $\overline{{\bf H}}$ is $R$-trivial.
\end{proof}

\subsection{The Tits--Weiss Conjecture}

\begin{theorem} 
Let $A$ be an Albert algebra defined over a field $K$ of an arbitrary 
characteristic. Then the group ${\bf Str}(A)(K)$ is generated by the
$U$-operators $U_a$ with $a\in A^\times$, and the scalar homotethies.
\end{theorem}
\begin{proof}
By the main result of the Appendix the Tits--Weiss conjecture holds
if and only if the Kneser--Tits conjecture holds for groups of type
$\mathrm{E}_{7,1}^{78}$ and $\mathrm{E}_{8,2}^{78}$. The result follows.
\end{proof}

\begin{remark} From this theorem, which we have now established in arbitrary characteristic, the $R$-triviality of ${\bf Str}(A)$ is immediate. Thus Theorem \ref{main} holds
in arbitrary characteristic as well.
\end{remark}

\subsection{Properties of the Functor of $R$-Equivalence Classes for Strongly Inner
Forms of Type $^1\mathrm{E}_6$} To a reductive algebraic group ${\bf G}$ over a field $K$
one can attach the functor of $R$-equivalence classes
$$
\mathcal{G}/R: Fields/K \longrightarrow Groups,\ \ \ F/K \to {\bf G}(F)/R
$$
where $Fields/K$ is the category of field extensions of $K$ and $Groups$ is
the category of abstract groups. The experts expect that the following conjectures
hold.

\medskip

\noindent
{\it Conjecture $1$. The functor $\mathcal{G}/R$ factors through the subcategory
of abelian groups of the category $Groups$, i.e.\ for all field extensions $F/K$
the group ${\bf G}(F)/R$ is abelian.}

\smallskip

\noindent
{\it Conjecture $2$. If $F$ is a finitely generated field over its prime subfield,
then ${\bf G}(F)/R$ is finite.}

\smallskip

\noindent
{\it Conjecture $3$. The functor $\mathcal{G}/R$ has transfers, i.e.\
there is a functorial collection of maps ${\rm tr}^F_K: {\bf G}(F)/R \to {\bf G}(K)/K$ 
for all finite field extensions $F/K$.}

\medskip

Furthermore, one expect that the norm principle holds for all semisimple
algebraic groups. 
Of course, all these conjectures are obviously true for $R$-trivial groups.
In particular they hold for rational groups. For instance, this is the case
for groups of type $\mathrm{G}_2$. In \cite{ACP} we investigated the case of groups
of type $\mathrm{F}_4$ arising from the first Tits construction. Here we consider
the next case of simple simply connected strongly inner forms  of type $\mathrm{E}_6$.

\begin{theorem} Let ${\bf G}$ be a simple simply connected strongly inner form of type $\mathrm{E}_6$
over a field $K$ of arbitrary characteristic. Then ${\bf G}(K)/R$ is an abelian group.
\end{theorem}
\begin{proof} 
The group ${\bf G}$ is the derived subgroup of the $R$-trivial group ${\bf Str}(A)$ for an Albert algebra $A$.
It follows that $[{\rm Str}(A),{\rm Str}(A)]\subset R{\bf G}(K)$, implying that
${\bf G}(K)/R$ is abelian.
\end{proof}  

\begin{theorem}
Let ${\bf G}$ be a simple simply connected strongly inner form of type $\mathrm{E}_6$
over a field $K$ of arbitrary characteristic. Then the functor $\mathcal{G}/R$
has transfers.
\end{theorem} 
\begin{proof} Apply Theorem $3.4$ in \cite{CM} to ${\bf G}'=R_{F/K}({\bf G})$.
\end{proof}

\begin{theorem}
Let ${\bf G}$ be a simple simply connected strongly inner form of type $\mathrm{E}_6$
over a field $K$ of arbitrary characteristic, and let ${\bf Z}\subset {\bf G}$ be its centre.
Then the norm principle holds for $({\bf Z},{\bf G})$.
\end{theorem}
\begin{proof} Since the corresponding adjoint group
$\overline{\bf G}$ is $R$-trivial, the result follows from Theorem~\ref{normprinciple}.
\end{proof}

\end{bibunit}

\newpage

\begin{bibunit}

\appendix
\section{\\ By Richard M.\ Weiss\\ \medskip {\rm \small Department of Mathematics\\
         Tufts University \\
         Medford, MA 02155, USA\\ rweiss@tufts.edu}}
In this appendix we examine the connection between groups $G(k)$ with index
$$
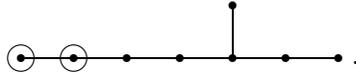
\begin{array}{ll}
\begin{picture}(125,25)
\put(00,02){\line(1,0){120}}
\put(80,02){\line(0,1){20}}
\put(00,02){\circle*{3}}
\put(20,02){\circle*{3}}
\put(40,02){\circle*{3}}
\put(60,02){\circle*{3}}
\put(80,02){\circle*{3}}
\put(100,02){\circle*{3}}
\put(120,02){\circle*{3}}
\put(80,22){\circle*{3}}
\put(00,02){\circle{10}}
\put(20,02){\circle{10}}
\end{picture}.
\end{array}
$$
(that is to say, $E_{8,2}^{78}$) and anisotropic exceptional cubic norm structures.
Our main goal is show the equivalence of Theorems~\ref{hex4} and~\ref{hex2}
below.

\begin{notation}\label{hex21}
Let $k$ be a field, 
let $G$ be a semi-simple
simply connected algebraic group of absolute type $E_8$ defined over $k$ such that 
the index of $G(k)$ is ${\mathscr I}:=E_{8,2}^{78}$,
let $S$ be a maximal $k$-split torus of $G$, let $T$ be
a maximal torus containing $S$ and defined over $k$, and 
let $\Phi$ be the root system of $G$ with respect to $T$.
The nodes of ${\mathscr I}$ form a root basis
of $\Phi$. Let $\tilde\alpha$ denote the highest
root with respect to this basis.
\end{notation}

\begin{notation}\label{hex3}
Let $\Phi_k$ denote the relative root system of $G$ with respect to $S$;
it is a root system of type~$G_2$. For each
$\alpha\in\,\Phi_k$, let $U_{(\alpha)}$ denote the unipotent $k$-subgroup defined in \cite[5.2]{BT}.
We call the groups $U_{(\alpha)}$ for $\alpha\in\,\Phi_k$ the {\it relative root groups} of $G$.
\end{notation}

Here now are the two assertions whose equivalence we want to demonstrate:

\begin{theorem}\label{hex4}
$G(k)=\langle U_{(\alpha)}(k)\mid\alpha\in\,\Phi_k\rangle$ for all $G(k)$ as in {\rm\ref{hex21}}.
\end{theorem}

\begin{theorem}\label{hex2}
Let $\Xi=(J,k,N,\#,T,\times,1)$ be an exceptional 
cubic norm structure and suppose that $\Xi$ is anisotropic.
Then the structure group ${\rm Str}(\Xi)$ of $\Xi$ is generated by the set 
\begin{equation}\label{hex2a}
\{U_a\mid a\in J^*\}\cup\{b\mapsto tb\mid t\in k^*\},
\end{equation}
where $U_a$ is as in {\rm\cite[(15.42)]{TW}}.
\end{theorem}

\begin{remark}\label{hex34}
All anisotropic cubic norm structures arise from
one of two constructions of Tits. A proof of this result, due to Racine and Petersson, can be found in
\cite[Chapters~15 and~30 and (17.6)]{TW}.
\end{remark}

Let $k$, $G$, ${\mathscr I}$, $S$, $\Phi$ and $\tilde\alpha$ be as in \ref{hex21}.
We begin our demonstration.

\begin{proposition}\label{hex1}
The  following hold:
\begin{enumerate}[\rm(i)]
\item The root group $U_{\tilde\alpha}$ of $G$ is defined over $k$.
\item The quotient $N_{G(k)}(U_{\tilde\alpha}(k))/C_{G(k)}(U_{\tilde\alpha}(k))$ acts freely
on the set of non-trivial elements of $U_{\tilde\alpha}(k)$.
\end{enumerate}
\end{proposition}

\begin{proof}
The two claims follow from the observation that the root $\tilde\alpha$ is orthogonal
to the subspace spanned by the nodes in the anisotropic part of ${\mathscr I}$.
\end{proof}

\begin{remark}\label{hex33}
Let $\Phi_k$ and $U_{(\alpha)}$ for $\alpha\in\Phi_k$ be as \ref{hex3}.
By \ref{hex1}(i), $U_{\tilde\alpha}$ is a relative root group of $G$ for some
long root of $\Phi_k$. Thus, in particular, $\dim_kU_{(\alpha)}=1$ 
for all long roots $\alpha$ of $\Phi_k$. 
\end{remark}

\begin{remark}\label{hex50}
We have
$$\sum_{\alpha\in\Phi_k}\dim_kU_{(\alpha)}=n_8-n_6,$$
where $n_\ell$ denotes the cardinality of a root system of type $E_\ell$ for $\ell=6$ and 8. 
It follows that $\dim_kU_{(\alpha)}=27$ for all the short roots $\alpha$ of $\Phi_k$.
\end{remark}

Now let ${\rm BldSc}(G(k))$ be the spherical building attached to $G(k)$ 
as described in \cite[42.3.6]{TW} and let $X$ denote the
Moufang hexagon associated with ${\rm BldSc}(G(k))$. Thus $X$ is the bipartite graph whose vertices
are the non-minimal parabolic subgroups of $G$ defined over $k$, 
where two of these parabolic subgroups are
adjacent whenever their intersection is also a parabolic subgroup defined over $k$.
Since the cocenter of $E_8$ is trivial, the center of $G$ is trivial. Thus 
$G(k)$ acts faithfully on ${\rm BldSc}(G(k))$ and hence
on $X$. From now on, we identify
$G(k)$ with its image in ${\rm Aut}(X)$.
For each $\alpha\in\Phi_k$, let $U_\alpha$ denote the subgroup $U_{(\alpha)}(k)$ of $G(k)$.

Let $\alpha_0,\alpha_1,\ldots,\alpha_{11}$ be a labeling of the twelve roots in $\Phi_k$ 
with subscripts in ${\mathbb Z}_{12}$ such that the angle between $\alpha_{i-1}$ and $\alpha_i$
is $\pi/6$ for each $i$ and $\alpha_i$ is long if and only if $i$ is even.
An apartment of $X$ is a circuit of length~$12$. 

\begin{proposition}\label{hex40}
There is a unique apartment $\Sigma$ of $X$
for which there is a labeling $x_0,x_1,\ldots,x_{11}$ of the vertices of $\Sigma$ with subscripts
in ${\mathbb Z}_{12}$ 
such that for each $i$, $x_{i-1}$ is adjacent to $x_i$ and 
$U_{\alpha_i}$ is the root group of $X$ 
corresponding to the root $(x_i,x_{i+1},\ldots,x_{i+6})$ of $\Sigma$ as defined in 
{\rm\cite[(4.1)]{TW}}.
\end{proposition}

\begin{proof}
This holds by \cite[Prop.~42.3.6]{TW}.
\end{proof}

\begin{proposition}\label{hex14}
Let $U_i=U_{\alpha_i}$ for each $i$. Then there
exists an anisotropic cubic norm structure
$$\Xi=(J,F,N,\#,T,\times,1)$$
and isomorphisms $x_i$ from the additive group of $F$ to $U_i$ for
$i=2$, $4$ and $6$ and isomorphisms $x_i$ from the additive group of $J$ to $U_i$
for $i=1$, $3$ and $5$ such that the commutator relations in \cite[(16.8)]{TW} 
hold and $[U_i,U_j]=1$ for all index pairs $(i,j)$ of indices with $1\le i<j\le 6$
that do not appear in this list of relations. 
\end{proposition}

\begin{proof}
This holds by \cite[(29.1)--(29.35)]{TW}. (Note that since $[U_{(\alpha)},U_{(\beta)}]=1$ for all 
long roots $\alpha,\beta$ of $\Phi_k$ at an angle of $\pi/3$,
the assumption that $V_i\ne1$ for all even $i$ at the top of \cite[page~303]{TW} is valid.)
\end{proof}

\begin{remark}\label{hex35}
By \cite[(7.5)]{TW}, $X$ is uniquely determined by $\Xi$. We can thus set
$X={\mathscr H}(\Xi)$ as in \cite[16.8]{TW}. By \cite[(35.13)]{TW}, 
${\mathscr H}(\Xi)\cong{\mathscr H}(\Xi')$
for two anisotropic cubic norm structures $\Xi$ and $\Xi'$ if and only if
$\Xi$ and $\Xi'$ are isotopes of each other.
\end{remark}

\begin{notation}\label{hex7}
Let $G_0={\rm Aut}(X)$, let $H_0$ denote the pointwise stabilizer of $\Sigma$ in $G_0$,
let $G_0^\dagger=\langle U_i\mid i\in{\mathbb Z}_{12}\rangle$,
let $H_0^\dagger=G_0^\dagger\cap H_0$ and let $J=G(k)\cap H_0$. 
\end{notation}
\noindent
Since $G_0^\dagger\subset
G(k)$, we have $H_0^\dagger\subset J$. In fact, we have:

\begin{proposition}\label{hex8}
$H_0^\dagger=J$ if and only if $G(k)=G_0^\dagger$.
\end{proposition}

\begin{proof}
Suppose that $H_0^\dagger=J$ and 
let $g\in G(k)$. By \cite[(4.12)]{TW}, there exists $a\in G_0^\dagger$ such that
$ga$ fixes both $\Sigma$ and the edge $\{x_0,x_1\}$. Hence $ga\in H_0$.
Since $G_0^\dagger\subset G(k)$, it follows that $ga\in J$. Hence $ga\in H_0^\dagger$
and thus $g\in G_0^\dagger$. Therefore $G(k)=G_0^\dagger$.
\end{proof}

\begin{proposition}\label{hex9}
Let $X_i=\langle\mu_i(a)\mu_i(b)\mid a,b\in U_i^*\rangle$ for $i=1$ and $6$, 
where $\mu_i$ is as defined in {\rm\cite[(6.1)]{TW}}. Then $H_0^\dagger=X_1X_6$.
\end{proposition}

\begin{proof}
This holds by \cite[(33.9)]{TW}.
\end{proof}

\begin{proposition}\label{hex10}
The image of $H_0^\dagger$
in ${\rm Aut}(U_6)$ is $\{x_6(t)\mapsto x_6(st)\mid s\in F^*\}$, where $x_6$ is as in 
{\rm\ref{hex14}}. In particular, the derived group of $H^\dagger_0$
centralizes $U_6$.
\end{proposition}

\begin{proof}
This holds by \cite[(33.16)]{TW}.
\end{proof}

\begin{proposition}\label{hex30}
The groups $J$ and $H_0^\dagger$ have the same image in ${\rm Aut}(U_6)$.
\end{proposition}

\begin{proof}
By \ref{hex10}, $H_0^\dagger$ acts transitively on $U_6^*$. 
By \ref{hex1}(ii) and \ref{hex33}, the image
of $J$ in ${\rm Aut}(U_6)$ acts freely on $U_6^*$. Since
$H_0^\dagger\subset J$, the claim follows.
\end{proof}

\begin{proposition}\label{hex11}
$H_0^\dagger=J$ if and only if $C_J(U_6)=C_{H_0^\dagger}(U_6)$.
\end{proposition}

\begin{proof}
Suppose that $C_J(U_6)=C_{H_0^\dagger}(U_6)$ and let $g\in J$.
By \ref{hex30}, there exists $h\in H_0^\dagger$ such that $gh\in C_J(U_6)$.
Hence $gh\in H_0^\dagger$ and thus $g\in H_0^\dagger$. Therefore $H_0^\dagger=J$.
\end{proof}

\begin{proposition}\label{hex12}
$C_{H_0}(U_6)$ acts faithfully on $U_1$.
\end{proposition}

\begin{proof}
This holds by \cite[(33.5)]{TW}.
\end{proof}

\begin{notation}\label{hex23}
Let $H$ be the anisotropic kernel of $G$. Thus $H$ is the derived group of the centralizer $C(S)$.
\end{notation}

\begin{proposition}\label{hex22}
$H(k)\subset C_J(U_6)$.
\end{proposition}

\begin{proof}
The fixed point set of $S(k)$ in $X$ is the apartment $\Sigma$. Hence 
$C(S)(k)\subset J$. By \ref{hex10} and \ref{hex30}, it follows that $H(k)\subset C_J(U_6)$.
\end{proof}

\begin{proposition}\label{hex31}
$F\cong k$ and $\dim_FJ=27$, where $F$ is as in {\rm\ref{hex14}}.
\end{proposition}

\begin{proof}
We give $U_1$ the structure of a vector space over $F$ by setting
$t\cdot x_1(a)=x_1(ta)$ for all $t\in F$ and all $a\in J$. Let $X_6$ be as in \ref{hex9}. 
By \cite[(29.20)]{TW}, the image of $X_6$ in ${\rm Aut}(U_1)$ is
$\{x_1(a)\mapsto t\cdot x_1(a)\mid t\in F^*\}$. The root groups $U_{\alpha_0}$ and $U_{\alpha_6}$
of $G$ are opposite and $X_6=\langle U_0,U_6\rangle\cap S(k)$. It follows
that there exists an isomorphism $\psi$ from $k$ to $F$ such that
$tu=\psi(t)\cdot u$ for all $t\in k$ and all $u\in U_1$. Hence $\dim_FJ=\dim_FU_1=27$
by \ref{hex50}.
\end{proof}

\begin{proposition}\label{hex15}
There is an isomorphism $\varphi$ from 
$C_J(U_6)$ to ${\rm Str}(\Xi)$ and $\Xi$ is exceptional.
\end{proposition}

\begin{proof}
By \ref{hex12} and the equation in the third display in \cite[(37.41)]{TW}, there exists
an isomorphism $\varphi$ from $C_{H_0}(U_6)$ to ${\rm Str}(\Xi)$ such that
$x_1(a)^h=x_1(a^{\varphi(h)})$ for all $a\in J$ and by \ref{hex22}, we have
$H(k)\subset C_J(U_6)\subset C_{H_0}(U_6)$. By \ref{hex31},
it follows that $\Xi$ is the Jordan $k$-algebra in \cite[42.5.6(d) and 42.6, Type~(8)]{TW}. 
Hence $\Xi$ is exceptional and $\varphi$ maps $H(k)$ to ${\rm Str}(\Xi)$ surjectively.
(See also the remarks that follow \cite[3.3.1]{strongly}.)
It follows that $H(k)=C_J(U_6)=C_{H_0}(U_6)$.
\end{proof}

\begin{proposition}\label{hex13}
$\varphi(C_{H_0^\dagger}(U_6))$ is the subgroup of ${\rm Str}(\Xi)$
generated by the set defined in \eqref{hex2a},
where $\varphi$ is as in {\rm\ref{hex15}}.
\end{proposition}

\begin{proof}
This holds by \cite[(33.16)]{TW}.
\end{proof}

\begin{proposition}\label{hex20}
Every anisotropic exceptional cubic norm structure arises from an application of 
{\rm\ref{hex14}} to the Moufang hexagon attached to a group $G(k)$ for some $G$ and some
$k$ satisfying the conditions in {\rm\ref{hex21}}.
\end{proposition}

\begin{proof}
This holds by \cite[42.6, Type~(8)]{TW}. 
\end{proof}

\begin{theorem}\label{hex16}
The assertions in {\rm\ref{hex4}} and {\rm\ref{hex2}} are equivalent.
\end{theorem}

\begin{proof}
By \ref{hex8} and \ref{hex11}, $G(k)=G_0^\dagger$ if and only if $C_{H_0^\dagger}(U_6)=C_J(U_6)$.
By \ref{hex15} and \ref{hex13}, $C_{H_0^\dagger}(U_6)=C_J(U_6)$ if and only if ${\rm Str}(\Xi)$
is generated by the set defined in \eqref{hex2a}.
The claim holds, therefore, by \ref{hex20}.
\end{proof}

\begin{remark}\label{hex24}
In \cite[Thm.~6.1]{thakur}, Thakur showed that \ref{hex4} holds in the case that the
cubic norm structure in \ref{hex14} is a first
Tits construction and in \cite[Thm.~7.2]{thakur}, he showed 
that the claim in \ref{hex2} holds for $\Xi$ a 
reduced (rather than anisotropic) exceptional cubic norm structures.
\end{remark}

\newpage
\end{bibunit}

\end{document}